\documentclass[reqno,10pt]{amsart}
\usepackage{CJK}
\usepackage{hyperref}
\usepackage{cite}
\usepackage{amsmath}
\usepackage{mathrsfs}
\usepackage{extarrows}
\usepackage{indentfirst}
\usepackage{color}

\makeatletter
\@namedef{subjclassname@2020}{%
	\textup{2020} Mathematics Subject Classification}
\makeatother

\numberwithin{equation}{section}

\newtheorem{theorem}{Theorem}[section]
\newtheorem{lemma}[theorem]{Lemma}
\newtheorem{definition}[theorem]{Definition}

\newtheorem{proposition}[theorem]{Proposition}
\newtheorem{remark}[theorem]{Remark}
\newtheorem{corollary}[theorem]{Corollary}

\allowdisplaybreaks

\newcommand{ \mint }{ {\int\hspace{-0.38cm}-}}

\begin{document}
	
	\title[\hfil Sobolev regularity for the nonlocal $(1, p)$-Laplace equations] {Sobolev regularity for the nonlocal $(1, p)$-Laplace equations in the superquadratic case}
	
		\author[D. Li and C. Zhang  \hfil \hfilneg]
		{Dingding Li  and Chao Zhang$^*$}
	
	\thanks{$^*$ Corresponding author.}
	
	\address{Dingding Li \hfill\break School of Mathematics, Harbin Institute of Technology, Harbin 150001, China}
	\email{a87076322@163.com}

	\address{Chao Zhang  \hfill\break School of Mathematics and Institute for Advanced Study in Mathematics, Harbin Institute of Technology, Harbin 150001, China}
	\email{czhangmath@hit.edu.cn}

	
	\maketitle
	
\begin{abstract}
We investigate the interior Sobolev regularity of weak solutions to the nonlocal $(1, p)$-Laplace equations in the superquadratic case $p\ge 2$. More precisely, for the fractional diﬀerentiability index $s_p\in \left(0, \frac{p-1}{p}\right]$, we establish that the weak solution $u$ exhibits $W_{\rm loc}^{\gamma, q}$-regularity for any $\gamma\in \left(0, \frac{s_p p}{p-1}\right)$ and $q\ge p$. When $s_p\in \left(\frac{p-1}{p}, 1\right)$, the gradient $\nabla u$ of the weak solution is shown to exist and belong to $\left( L^q_\mathrm{loc}(\Omega)\right)^N$. As a product, the explicit H\"{o}lder continuity estimates of weak solutions are derived. The proof relies on a detailed analysis of the structural characteristics of 
$(1, p)$-growth in the nonlocal setting, combined with the finite difference quotient method, tail estimates, refined energy estimates, and a Moser-type iteration scheme.
\end{abstract}

\section{Introduction}
Let $\Omega$ be a bounded domain in $\mathbb{R}^N$ with $N\ge2$. In this paper, we consider the following nonlocal $(1, p)$-Laplace equation 
\begin{align}
	\label{1.1}
		(-\Delta_1)^{s_1}u+(-\Delta_p)^{s_p}u=0\quad\text{in }\Omega,
\end{align}
where $s_1, s_p\in (0, 1)$, $p\ge 2$ and the fractional operators $(-\Delta_1)^{s_1}$ and $(-\Delta_p)^{s_p}$ are defined by
\begin{align*}
	(-\Delta_1)^{s_1}u(x):=2\mathrm{P.V.}\int_{\mathbb{R}^N}\frac{u(x)-u(y)}{|u(x)-u(y)|}\frac{dy}{|x-y|^{N+s_1}}
\end{align*}
and
\begin{align*}
	(-\Delta_p)^{s_p}u(x):=2\mathrm{P.V.}\int_{\mathbb{R}^N}\frac{|u(x)-u(y)|^{p-2}(u(x)-u(y))}{|x-y|^{N+s_p p}}\,dy
\end{align*}
for $x\in \Omega$. Here $\mathrm{P.V.}$ denotes the integral taken in the principal value sense. 

The equations of the type considered in \eqref{1.1} were recently proposed by G\'{o}rny, Maz\'{o}n and Toledo \cite{GMT24}
in a general setting of random walk spaces. These equations incorporate two random walk structures so that the associated functionals have diﬀerent growth on each structure, which can be regarded as a special case of  $(p, q)$-growth problems. Since the pioneering works \cite{M89, M91}, this class of problems has been extensively investigated. We refer the readers to \cite{MR21, C18, BM20, CMM24, M21} and the references therein for detailed overviews.

\subsection{Overview of related results.} Let us start with the problems with local $1$-growth. For the inhomogeneous $1$-Laplace Dirichlet problem
\begin{align}
	\label{0.1}
	\begin{cases}
		-\Delta_1 u:=-\mathrm{div}\left( \frac{\nabla u}{|\nabla u|}\right)=f & \textmd{in } \Omega,\\
		u=0& \textmd{on }  \partial\Omega,
	\end{cases}
\end{align}
it is well-known that the diffusion singularity of the 1-Laplacian $\Delta_1$ manifests prominently on the set $\left\lbrace Du=0\right\rbrace $, commonly referred to as a facet. To address this issue, a vector field $Z\in L^\infty(\Omega; \mathbb{R}^N)$ was  introduced by using the Anzellotti-Frid-Chen pairing theory (see \cite{A83, CF99, CF01, CF03}) that serves as a generalized representative of the quotient $\frac{Du}{|Du|}$ imprecisely. We refer to a series of works by Maz\'{o}n and collaborators \cite{MMT23, HM22, MRS14, MRS15} for more comprehensive treatments of this approach. Moreover, problem \eqref{0.1} is usually addressed as the limit version for $p$-Laplace equations as $p$ goes to $1^+$. The behavior of solutions to \eqref{0.1} has been rigorously analyzed under various assumptions on the right-hand side term $f$, see \cite{CT03, K90, MST08}.  In these works, a smallness condition on $f$ is typically imposed to ensure that the solution remains finite almost everywhere in $\Omega$. However, within the $L^1$ framework, one cannot generally expect that the solutions are finite almost everywhere in $\Omega$.  In this context, Mercaldo, Segura de Le\'{o}n and Trombetti \cite{MST09} established the existence of renormalized solutions. 

Regarding the regularity of classical 1-growth problems, as far as we know, since the ellipticity of 1-Laplacian becomes singular over the facet $\left\lbrace Du=0\right\rbrace$, the solutions of $-\Delta_1 u=f$ may fail to be continuously differentiable. Nevertheless, the results on problems with nearly linear growth stand in sharp contrast to those in the 1-growth regime. In particular, De Filippis and Mingione \cite{DM23} established the Schauder-type estimates for minimizers of the functional:
\begin{align*}
	\mathcal{L}(w, \Omega):=\int_{\Omega}\left[c(x)|Dw|\log(1+|Dw|)+a(x)|Dw|^q\right]\,dx, \quad q>1.
\end{align*}
Under suitable conditions on the coefficients $a(x)$ and $c(x)$, these results yield $C^{1, \alpha}$-regularity for the corresponding minimizers. We also refer to \cite{DDP24, DP24, FS98, FS99, EM00, MS99} for more results on the problems with nearly linear growth.

When it comes to the nonlocal counterpart
\begin{align*}
	(-\Delta_1)^{s} u=0,
\end{align*}
there remain relatively few results addressing problems with 1-growth.  It is worth noting that Maz\'{o}n et al. have made a significant contribution to this field. In their seminal works \cite{MRT16, MPRT16},  
they not only introduced a rigorous definition of weak solutions for nonlocal evolution problems with 1-growth but also characterized the new qualitative phenomena inherent in these problems. In contrast to the local case, where the expression $\frac{Du}{|Du|}$ arises naturally, the nonlocal counterpart takes the form $\frac{u(x)-u(y)}{|u(x)-u(y)|}$. To properly interpret this quantity, one must identify a function $Z$ representing the nonlocal quotient. In this direction, Andreu et al. \cite{AMRT08, AMRT09} demonstrated that for a class of nonlocal diffusion equations, the corresponding quantity satisfies $Z(x, y, t)\in\mathrm{sgn}(u(x,t)-u(y,t))$, where $\mathrm{sgn}(\cdot)$ denotes the set-valued sign function with $\mathrm{sgn}(0)=[-1, 1]$ (see also \cite{MRT19}). 
Bucur et al. \cite{BDLM23} investigated the existence of $(s, 1)$-harmonic function and established their equivalence to the minimizers of functional
\begin{align*}
	\int_{\mathbb{R}^N}\int_{\mathbb{R}^N}\frac{|u(x)-u(y)|}{|x-y|^{N+s}}\,dxdy.
\end{align*}
Assuming the inhomogeneous term $f$ belongs to $L^\frac{N}{s}(\Omega)$ with a smallness assumption, Bucur \cite{B23} studied the minimizer of $\mathcal{F}^{s_p}_p$ given by
\begin{align*}
	\mathcal{F}^{s_p}_p(u):=\frac{1}{2p}\int_{\mathbb{R}^N}\int_{\mathbb{R}^N}\frac{|u(x)-u(y)|^p}{|x-y|^{N+s_pp}}\,dxdy-\int_{\Omega}fu\,dx,
\end{align*}
where $p\ge1$ and $s_p=N+s-\frac{N}{p}\in[s,1)$. And it is shown that the minimizers correspond to weak solutions of the nonlocal 1-Laplacian problem and a ``flatness" result was obtained, which implies that the expression of $Z$ towards $\frac{u(x)-u(y)}{|u(x)-u(y)|}$ may not precise.

Under certain natural structural conditions, Bucur et al. \cite{BDLV25} proved that any $s$-minimal function $u$ lies in $C(\Omega)\cap L^\infty(\Omega)$ and admits a continuous extension $\overline{u}\in C(\overline{\Omega})$. For further regularity results, we refer to the work of Novaga and Onoue \cite{NO23}, who investigated the minimizers of a nonlocal variational problem arising from an image denoising model:
\begin{align*}
	\text{min}\left\lbrace \mathcal{F}_{K,f}(u)\big|u\in BV_K(\mathbb{R}^N)\cap L^2(\mathbb{R}^N)\right\rbrace ,
\end{align*}
where $\mathcal{F}_{K,f}$ is defined by
\begin{align*}
	\mathcal{F}_{K,f}(u):=\frac{1}{2}\int_{\mathbb{R}^N\times\mathbb{R}^N}K(|x-y|)|u(x)-u(y)|dxdy+\frac{1}{2}\int_{\mathbb{R}^N}\left( u(x)-f(x)\right)^2dx 
\end{align*}
and the space $BV_K(\mathbb{R}^N)$ is given by
\begin{align*}
	BV_K(\mathbb{R}^N):=\left\lbrace u\in L^1(\mathbb{R}^N)\bigg|\int_{\mathbb{R}^N\times\mathbb{R}^N}K(|x-y|)|u(x)-u(y)|dxdy<+\infty\right\rbrace 
\end{align*}
with $K$ being a kernel singular at the origin. In two dimensions, under the assumption that the original image $f$ is locally $\beta$-H\"{o}lder continuous with $\beta\in (1-s,1]$, the authors showed that the minimizers preserve the same local H\"{o}lder regularity.

Finally, we would like to turn to the perturbed 1-Laplace equation of the form
\begin{align*}
	-\mathrm{div}\left( \frac{\nabla u}{|\nabla u|}\right)-\mathrm{div}\left( |\nabla u|^{p-2}\nabla u \right) =f\quad\text{in }\Omega.
\end{align*}
For the cases $p=2$ and $p=3$, this equation arises in fluid mechanics \cite{DL76} and materials science \cite{S93} respectively. Such problems have been investigated only in a limited number of specific settings. 
The anisotropic case involving the $p$-Laplacian when it has linear growth on some of the coordinates has also been studied, for instance, in \cite{G23, MRST10}. Employing De Giorgi-Nash-Moser theory, Tsubouchi \cite{T21} proved the local Lipschitz continuity of weak solutions. Under the additional assumption of convexity for weak solutions, Giga and Tsubouchi \cite{GT22} showed that the solutions are actually of $C^1$-regular. Subsequently, Tsubouchi \cite{T24} removed the convexity assumption and achieved $C^1_\mathrm{loc}$-regularity by employing De Giorgi's truncation method and the freezing coefficients technique. These methods were further extended to prove regularity results for parabolic $(1, p)$-Laplace system in \cite{T25, T24A, T25A}.

\subsection{Main results.} Inspired by the aforementioned literature, we aim to investigate the regularity of weak solutions to \eqref{1.1}.  Although \cite{GMT24} has formulated a precise notion of weak solutions and systematically analyzed several critical phenomena associated with nonlocal $(1, p)$-growth evolution problems, the fundamental regularity of this model was still unknown. To this end, thanks to the development of the finite difference quotient technique within the fractional framework (see \cite{BL17, BLS18, BDL2401, BDL2402, FLZ25, DKLN23, GL24}), we establish the local Sobolev regularity of weak solutions to \eqref{1.1}.

Different from the standard nonlocal $p$-growth case, the tail estimates developed in \cite{BDL2402} fail to yield a suitable energy inequality in our setting, primarily due to the lack of regularity information on the function $Z$. As a result, when dealing with the energy inequality, we have to control the local part of the $p$-growth using the $1$-growth. This leads us to distinguish between two regimes for the parameter $s_p$, namely $s_p\in \left( 0,\frac{p-1}{p}\right] $ and $s_p\in \left( \frac{p-1}{p}, 1\right)$, which are consistent with the regularity findings in \cite{DKP16, BL17}. Additionally, the local boundedness will be established using techniques developed in \cite{DFZ24, HS14, C17, DM24}.

\smallskip
Before presenting the main results, we give the definition of weak solutions to \eqref{1.1}.

\begin{definition}
	\label{def1}
A function $u\in W^{s_1,1}_{\mathrm{loc}}(\Omega)\cap W^{s_p,p}_{\mathrm{loc}}(\Omega)\cap L^{p-1}_{s_pp}(\mathbb{R}^N)$ is said to be a weak solution to problem \eqref{1.1} if
	\begin{itemize}
		\item [(i)] there exists a function $Z\in L^\infty(\mathbb{R}^N\times\mathbb{R}^N)$ satisfying $Z\in \text{sgn}(u(x)-u(y))$;
		\item [(ii)] for any $\varphi\in W^{s_1,1}_{0}(\Omega)\cap W^{s_p,p}_{0}(\Omega)$, there holds
		\begin{align}
			\label{1.2}
			0&=\int_{\mathbb{R}^N}\int_{\mathbb{R}^N}Z\frac{\varphi(x)-\varphi(y)}{|x-y|^{N+s_1}}\,dxdy\nonumber\\
			&\quad+\int_{\mathbb{R}^N}\int_{\mathbb{R}^N}\frac{|u(x)-u(y)|^{p-2}(u(x)-u(y))(\varphi(x)-\varphi(y))}{|x-y|^{N+s_pp}}\,dxdy.
		\end{align}
	\end{itemize}
\end{definition}

\begin{remark}
	\label{rem0}
	To ensure the existence of weak solutions, a decay assumption is required for the $p$-growth term, i.e.,
	\begin{align*}
		u\in L^{p-1}_{s_pp}(\mathbb{R}^N).
	\end{align*}
	However, due to the integrability property
	\begin{align}
		\label{assume1}
		\int_{\mathbb{R}^N\backslash B_R(x_0)}\frac{dy}{|x-x_0|^{N+s_1}}<+\infty,
	\end{align}
	no decay assumption is needed for the 1-growth term. The existence of weak solutions to problem \eqref{1.1} can be established via the direct method used in \cite{B23}.
\end{remark}

Define
\begin{align*}
	\mathcal{T}:=\|u\|_{L^\infty(B_{R}(x_0))}+\mathrm{Tail}(u;x_0,R),
\end{align*}
where $\operatorname{Tail}(u; x_0, R)$ denotes the tail term associated with the function $u$ at the point $x_0$ and radius $R$ (see Section \ref{sec2} for details). We are now ready to state the main results of this work. 

\begin{theorem}
	\label{th12}
	Let $p\ge2$, $s_p\in \left( 0,\frac{p-1}{p}\right]$ and let $u$ be a locally bounded weak solution to problem \eqref{1.1} in the sense of Definition \ref{def1}, we have
	\begin{align*}
		u\in W^{\gamma,q}_\mathrm{loc}(\Omega)
	\end{align*}
	for any $q\ge p$ and $\gamma\in \left( 0,\frac{s_pp}{p-1}\right) $. Moreover, there exist two constants $C$ and $\kappa$ depending on $N, p, q, s_1, s_p, \gamma$ such that for any ball $B_R\equiv B_R(x_0)\subset\subset\Omega$ with $R\in (0,1)$ and any $r\in (0,R)$,
	\begin{align*}
		[u]^q_{W^{\gamma,q}(B_r)}\le \frac{C\left( \mathcal{T}+[u]_{W^{s_p,p}(B_R)}+1\right)^q}{(R-r)^\kappa}.
	\end{align*}
\end{theorem}

\begin{theorem}
	\label{th26}
	Let $p\ge2$, $s_p\in \left( \frac{p-1}{p}, 1\right) $ and let $u$ be a locally bounded weak solution to problem \eqref{1.1} in the sense of Definition \ref{def1}, we have
	\begin{align*}
		u\in W^{1,q}_\mathrm{loc}(\Omega)
	\end{align*}
	for any $q\ge p$. Moreover, there exist two constants $C$ and $\kappa$ depending on $N, p, q, s_1, s_p, \gamma$ such that for any ball $B_R\equiv B_R(x_0)\subset\subset\Omega$ with $R\in(0,1)$ and any $r\in(0,R)$,
	\begin{align*}
		\|\nabla u\|^q_{L^q(B_r)}\le \frac{C\left( \mathcal{T}+[u]_{W^{s_p,p}(B_R)}+1\right)^q }{(R-r)^{\kappa}}.
	\end{align*}
\end{theorem}

By applying the Morrey-type embedding for fractional Sobolev spaces, Theorems \ref{th12} and \ref{th26} directly yield the following local H\"{o}lder regularity of weak solutions.

\begin{corollary}
	\label{cor13}
	Let $p\ge2$ and $s_p\in\left( 0,\frac{p-1}{p}\right] $. Then, for any locally bounded weak solution to problem \eqref{1.1} in the sense of Definition \ref{def1}, we have
	\begin{align*}
		u\in C^{0, \gamma}_\mathrm{loc}(\Omega)
	\end{align*}
	for any $\gamma\in \left( 0,\frac{s_p p}{p-1}\right)$. Moreover, there exist two constants $C$ and $\kappa$ depending on $N, p, s_1, s_p, \gamma$, such that for any ball $B_R\equiv B_R(x_0)\subset\subset\Omega$ with $R\in (0,1)$ and any $r\in(0,R)$,
	\begin{align*}
		[u]_{C^{0,\gamma}(B_r)}\le \frac{C\left( \mathcal{T}+[u]_{W^{s_p,p}(B_R)}+1\right) }{(R-r)^\kappa}.
	\end{align*}
\end{corollary}

\begin{corollary}
	\label{cor27}
	Let $p\ge2$ and $s_p\in \left( \frac{p-1}{p},1\right) $.  Then, for any locally bounded weak solution to problem \eqref{1.1} in the sense of Definition \ref{def1}, we have
	\begin{align*}
		u\in C^{0,\gamma}_\mathrm{loc}(\Omega)
	\end{align*}
	for any $\gamma\in(0,1)$. Moreover, there exist two constants $C$ and $\kappa$ depending on $N, p, s_1, s_p, \gamma$, such that for any ball $B_R\equiv B_R(x_0)\subset\subset\Omega$ with $R\in(0,1)$ and any $r\in(0,R)$,
	\begin{align*}
		[u]_{C^{0,\gamma}(B_r)}\le \frac{C\left( \mathcal{T}+[u]_{W^{s_p,p}(B_R)}+1\right) }{(R-r)^{\kappa}}.
	\end{align*}
\end{corollary}

\begin{remark}
	\label{rem1}
It is worth emphasizing that the appearance of the constant “1” in the inequalities of the main results stems directly from \eqref{assume1}. Theorems \ref{th12} and \ref{th26} reveal distinct arguments for the cases $s_p\in \left( 0,\frac{p-1}{p}\right] $ and $s_p\in \left( \frac{p-1}{p}, 1\right) $, with the threshold $\frac{p-1}{p}$ characterizing the existence of $\nabla u$. Indeed, we define the following sequence
	\begin{align*}
		\gamma_0=s_p, \quad\gamma_{i+1}=\frac{s_pp}{q}+\gamma_i\frac{q-p+1}{q},
	\end{align*}
from which it follows that
\begin{itemize}
	\item For $s_p \in \left(0,  \frac{p-1}{p} \right]$, $\gamma_i \to \frac{s_p p}{p-1}$  as  $i \to \infty$.
	
	\smallskip
	
	\item For  $s_p \in \left( \frac{p-1}{p}, 1 \right)$, there exists an index  $i_0$  such that  $\gamma_i > 1$  for all  $i \geq i_0$,  thereby establishing the  $W^{1, q}$-regularity of weak solutions.
\end{itemize}

The motivation for introducing the sequence $\{\gamma_i\}$ is detailed in Lemmas \ref{lem10} and \ref{lem16}, and the iterative process yielding the limit parameter  $\min\left\{ 1, \frac{s_p p}{p-1} \right\}$ is presented in Lemma \ref{lem11} and Proposition \ref{pro17}. While, for $s_p \in \left( \frac{p-1}{p}, 1 \right)$,  the fractional Sobolev regularity of  $\nabla u$  is established, with the precise statement provided in Lemma \ref{lem22}.
\end{remark}

This paper is organized as follows. In Section \ref{sec2}, we introduce the notations for fractional Sobolev spaces and state several critical lemmas. Section \ref{sec3} is dedicated to establishing the local boundedness of weak solutions to problem \eqref{1.1}. Finally, Sections \ref{sec4} and \ref{sec5} address the Sobolev regularity of weak solutions in the cases that $s_p \in \left( 0, \frac{p-1}{p} \right]$ and  $s_p \in \left( \frac{p-1}{p}, 1 \right)$, respectively.

\section{Preliminaries}
\label{sec2}

In this section, we introduce notations that will be used later. For any $z \in \mathbb{R}^N$, let $|z|$ denote its Euclidean norm. Throughout the paper, we denote by $C$  generic constants whose values may vary even within the same inequality. When necessary, we explicitly indicate their dependencies, for instance, if a constant depends on $N, p, s_1, s_p$, it will be written as $C(N, p, s_1, s_p)$.

For a function $w: \Omega\rightarrow \mathbb{R}$ and $\alpha\in (0, 1)$ we denote by $w\in C^{0, \alpha}_\mathrm{loc}(\Omega)$ that $w\in C^{0, \alpha}\left(B_R(x_0)\right)$ for any ball $B_R(x_0)\subset\subset \Omega$. The semi-norm of $C^{0, \alpha}$ is defined by
\begin{align*}
	[w]_{C^{0, \alpha}\left(B_R(x_0)\right)}:=\sup\limits_{x, y\in B_R(x_0),  x\neq y}\frac{|w(x)-w(y)|}{|x-y|^\alpha}.
\end{align*}

We introduce the (fractional) Sobolev space $W^{1, q}(\Omega)$ and $W^{\gamma, q}$, where $q\ge1$ and $\gamma\in(0, 1)$. These spaces are respectively defined as

\begin{align*}
	\left\lbrace u\in L^q(\Omega)\bigg|\|u\|_{W^{1,q}(\Omega)}:=\|u\|_{L^q(\Omega)}+\|\nabla u\|_{L^q(\Omega)}<+\infty\right\rbrace 
\end{align*}
and
\begin{align*}
	\left\lbrace u\in L^q(\Omega)\bigg|[u]_{W^{\gamma,q}(\Omega)}:=\left( \int_{\Omega}\int_{\Omega}\frac{|u(x)-u(y)|^q}{|x-y|^{N+\gamma q}}\,dxdy\right)^\frac{1}{q} <+\infty\right\rbrace .
\end{align*}

For $q,\gamma>0$, the tail space $L^q_\gamma(\mathbb{R}^N)$ contains all function $w\in L^q_\mathrm{loc}(\mathbb{R}^N)$ satisfying
\begin{align*}
	\int_{\mathbb{R}^N}\frac{|w|^q}{(1+|x|)^{N+\gamma}}\,dx<+\infty
\end{align*}
and we define $\mathrm{Tail}(u;x_0,R)$ for a function $u$ belonging to the space $L^{p-1}_{s_pp}(\mathbb{R}^N)$ as
\begin{align*}
	\mathrm{Tail}(u;x_0,R):=\left( R^{s_pp}\int_{\mathbb{R}^N\backslash B_R(x_0)}\frac{|u|^{p-1}}{|x-x_0|^{N+s_p p}}\,dx\right)^\frac{1}{p-1}. 
\end{align*}

Under the assumption that $u$ is locally bounded in $\Omega$, we have the following estimate.

\begin{lemma}
	\label{lem4}
	Let $p>1$ and $s_p\in (0, 1)$. For any $u\in L^{p-1}_{s_pp}(\mathbb{R}^N)$ and any ball $B_R\equiv B_R(x_0)\subset\subset\Omega$, $r\in (0,R)$, we have
	\begin{align*}
		\mathrm{Tail}(u;x_0,r)^{p-1}\le C(N)\left( \frac{R}{r}\right)^N\left( \mathrm{Tail}(u;x_0,R)+\|u\|_{L^\infty(B_R)}\right)^{p-1}. 
	\end{align*}
\end{lemma}

Next, for $\gamma>0$ and $a\in\mathbb{R}$, we define
\begin{align*}
	J_\gamma(a):=|a|^{\gamma-2}a.
\end{align*}
We summarize two algebraic inequalities from \cite{BDL2402}, which will be used in Sections \ref{sec4} and \ref{sec5}.

\begin{lemma}
	\label{lem2}
	For any $\gamma>0$ and any $a,b\in \mathbb{R}$, we have
	\begin{align*}
		C_1\left( |a|+|b|\right) ^{\gamma-1}|b-a|\le \left| J_{\gamma+1}(b)-J_{\gamma+1}(a)\right| \le C_2\left( |a|+|b|\right) ^{\gamma-1}|b-a|
	\end{align*}
	where $C_1=\min\left\lbrace \gamma,2^{1-\gamma}\right\rbrace $, $C_2=\max\left\lbrace \gamma,2^{1-\gamma}\right\rbrace $.
\end{lemma}

\begin{lemma}
	\label{lem3}
	Let $p\ge2$ and $\delta\ge1$. Then there exists a constant $C$ such that
	\begin{align*}
		&\quad\left( J_p(a-b)-J_p(c-d)\right)\left( J_{\delta+1}(a-c)e^2-J_{\delta+1}(b-d)f^2\right)\nonumber\\ 
		&\ge \frac{1}{C}\left( |a-b|+|c-d|\right)^{p-2}\left( |a-c|+|b-d|\right)^{\delta-1}|(a-c)-(b-d)|^2(e^2+f^2)\nonumber\\
		&\quad-C \left( |a-b|+|c-d|\right)^{p-2}\left( |a-c|+|b-d|\right)^{\delta+1}|e-f|^2 
	\end{align*}
	for any $a,b,c,d\in \mathbb{R}$ and $e,f\in \mathbb{R}^+$.
\end{lemma}

We now recall a classical iteration lemma that comes from \cite[Lemma 4.3]{HS14} (see also \cite[Lemma 2.1]{VZ10}) and provide several Sobolev embedding inequalities. Lemmas \ref{lem23}--\ref{lemMorrey1} are derived from \cite{DPV12} and Lemma \ref{Morrey2} is standard.

\begin{lemma}
	\label{lemY}
	Let $\left\lbrace Y_j\right\rbrace_{j\in\mathbb{N}} $ be a sequence of positive numbers, satisfying the recursive inequalities
	\begin{align*}
		Y_{j+1}\le Kb^jY_j^{1+\delta}
	\end{align*}
	for all $j=0,1,2,\dots$ with $K>0,b>1$ and $\delta>0$ are given numbers. If
	\begin{align*}
		Y_0\le \min\left\lbrace 1,(2K)^{-\frac{1}{\delta}}b^{-\frac{1}{\delta^2}}\right\rbrace, 
	\end{align*}
	we have $Y_j\le 1$ for some $j\in\mathbb{N}$. Moreover,
	\begin{align*}
		Y_j\le\min\left\lbrace 1,(2K)^{-\frac{1}{\delta}}b^{-\frac{1}{\delta^2}}b^{-\frac{j}{\delta}}\right\rbrace 
	\end{align*}
	for any $j\ge j_0$, where $j_0$ is the smallest $j\in \mathbb{N}$ that satisfies $Y_j\le 1$. In particular, $Y_j$ converges to $0$ as $j\rightarrow+\infty$. 
\end{lemma}

\begin{lemma}[\text{Embedding $W^{\gamma,q}\hookrightarrow L^{\frac{Nq}{N-\gamma q}}$}]
	\label{lem23}
	Let $q\ge1$, $\gamma\in(0,1)$ such that $\gamma q<N$. Then, for any $w\in W^{\gamma,q}(B_R)$, we have
	\begin{align*}
		\left[ \mint_{B_R}\left| w-(w)_{B_R}\right|^\frac{Nq}{N-\gamma q}\,dx \right]^{\frac{N-\gamma q}{Nq}}\le CR^\gamma\left[ \int_{B_R}\mint_{B_R}\frac{|w(x)-w(y)|^q}{|x-y|^{N+\gamma q}}\,dxdy\right]^\frac{1}{q}  
	\end{align*}
	with $C=C(N,\gamma,q)$.
\end{lemma}

\begin{lemma}[\text{Embedding $W^{1,q}\hookrightarrow W^{\gamma,q}$}]
	\label{lem18}
	Let $q\ge1$ and $\gamma\in(0,1)$. Then, for any $w\in W^{1,q}(B_R)$, we have
	\begin{align*}
		\int_{B_R}\int_{B_R}\frac{|w(x)-w(y)|^q}{|x-y|^{N+\gamma q}}\,dxdy\le C(N)\frac{R^{(1-\gamma)q}}{(1-\gamma)q}\int_{B_R}|\nabla w|^q\,dx.
	\end{align*}
\end{lemma}

\begin{lemma}[\text{Embedding $W^{\gamma,q}\hookrightarrow C^{0,\gamma-\frac{N}{q}}$}]
	\label{lemMorrey1}
	Let $q\ge1$ and $\gamma\in (0,1)$ such that $\gamma q >N$. Then, there exists a constant $C=C(N, q, \gamma)$ such that for any $w\in W^{\gamma,q}(B_R)$, we have
	\begin{align*}
		[w]_{C^{0,\gamma-\frac{N}{q}}(B_R)}\le C[w]_{W^{\gamma, q}(B_R)}.
	\end{align*}
\end{lemma}

\begin{lemma}[\text{Embedding $W^{1,q}\hookrightarrow C^{0,1-\frac{N}{q}}$}]
	\label{Morrey2}
	Let $q> N$, there exists a constant $C=C(N, q)$ such that for any $w\in W^{1,q}(B_R)$,
	\begin{align*}
		[w]_{C^{0,1-\frac{N}{q}}(B_R)}\le C\|\nabla w\|_{L^q(B_R)}.
	\end{align*}
\end{lemma}

We end the section with technical lemmas that are essential for applying the finite difference quotient technique to study the local Sobolev regularity of weak solutions. All the following lemmas have been summarized in \cite{GT83, BL17, G15, DM2301, DKLN23, D04, L19}. For a direction vector $h\in \mathbb R^N$ and measurable $u\in \Omega\to \mathbb R$, define the difference operator 
\begin{align*}
	\tau_hu(x):=u(x+h)-u(x), \quad x\in\mathbb{R}^N.
\end{align*}

\begin{lemma} 
	\label{lem5}
	Let $q>1$, $\gamma\in (0,1)$, $R>0$ and $d\in (0,R)$. Then, there exists a constant $C=C(N,q)$ such that for any $w\in W^{\gamma,q}(B_R)$,
	\begin{align*}
		\int_{B_{R-d}}|\tau_hw|^q\,dx\le C|h|^{\gamma q}\left[ (1-\gamma)[w]^q_{W^{\gamma,q}(B_R)}+\left( \frac{R^{(1-\gamma)q}}{d^q}+\frac{1}{\gamma d^{\gamma q}}\right)\|w\|^q_{L^q(B_R)} \right] 
	\end{align*}
	for any $h\in B_d\backslash \left\lbrace 0\right\rbrace$.
\end{lemma}

\begin{lemma}
	\label{lem20}
	Let $q\ge1$ and $0<d<R$, for any $w\in W^{1,q}_\mathrm{loc}(B_R)$ and any $h\in B_d\backslash\left\lbrace 0\right\rbrace $, we have
	\begin{align*}
		\|\tau_hu\|_{L^q(B_{R-d})}\le C(N)|h|\|Dw\|_{L^q(B_R)}.
	\end{align*}
\end{lemma}

\begin{lemma}
	\label{lem6}
	Let $q\ge1$, $\gamma>0$, $M\ge0$, $0<r<R$ and $d\in \left( 0,\frac{1}{2}(R-r)\right] $. Then, there exists a constant $C=C(q)$ such that whenever $w\in L^q(B_R)$ satisfies
	\begin{align*}
		\int_{B_r}|\tau_h(\tau_hw)|^q\,dx\le M^q|h|^{\gamma q}
	\end{align*}
	for any $h\in B_d\backslash \left\lbrace 0\right\rbrace $. Then in the case $\gamma\in (0,1)$, we have, for any $h\in B_{\frac{1}{2}d}\backslash\left\lbrace 0\right\rbrace $,
	\begin{align*}
		\int_{B_r}|\tau_hw|^q\le C(q)|h|^{\gamma q}\left[ \left( \frac{M}{1-\gamma}\right)^q+\frac{1}{d^{\gamma q}}\int_{B_R}|w|^q\,dx \right].
	\end{align*}
	In the case $\gamma>1$, we have
	\begin{align*}
		\int_{B_r}|\tau_hw|^q\,dx\le C(q)|h|^q\left[ \left( \frac{M}{\gamma-1}\right)^qd^{(\gamma-1)q}+\frac{1}{d^q}\int_{B_R}|w|^q\,dx \right].
	\end{align*}
	In the case $\gamma=1$, we have
	\begin{align*}
		\int_{B_r}|\tau_hw|^q\,dx\le C(q)|h|^{\lambda q}\left[ \left( \frac{M}{1-\lambda}\right)^qd^{(1-\lambda)q}+\frac{1}{d^{\lambda q}}\int_{B_R}|w|^q\,dx \right] 
	\end{align*}
	for any $0<\lambda<1$.
\end{lemma}

\begin{lemma}
	\label{lem7}
	Let $q\ge1$, $\gamma\in (0,1]$, $M\ge1$ and $0<d<R$. Then, there exists a constant $C=C(N,q)$ such that whenever $w\in L^q(B_{R+d})$ satisfies
	\begin{align*}
		\int_{B_R}|\tau_hw|^q\,dx\le M^q|h|^{\gamma q}
	\end{align*}
	for any $h\in B_d\backslash\left\lbrace 0\right\rbrace $, $w\in W^{\beta,q}(B_R)$ for any $\beta\in (0,\gamma)$. Moreover, we have
	\begin{align*}
		[w]^q_{W^{\beta,q}(B_R)}\le C\left[ \frac{d^{(\gamma-\beta)q}}{\gamma-\beta}M^q+\frac{1}{\beta d^{\beta q}}\|w\|^q_{L^q(B_R)}\right]. 
	\end{align*}
\end{lemma}

\begin{lemma}
	\label{lem14}
	Let $q>1$, $M>0$ and $0<d<R$. Then, any $w\in L^q(B_R)$ that satisfies
	\begin{align*}
		\int_{B_{R-d}}|\tau_hw|^q\,dx\le M^q|h|^q
	\end{align*}
	for any $h\in B_d\backslash \left\lbrace 0\right\rbrace $ is weakly differentiable in $B_{R-d}$. Moreover, we have
	\begin{align*}
		\int_{B_{R-d}}|Du|^q\,dx\le C(N)M^q.
	\end{align*} 
\end{lemma}

\begin{lemma}
	\label{lem21}
Let $q\ge1$, $\gamma\in(0, 1)$, $M>0$, $R>0$ and $d\in(0,R)$.  For any $w\in W^{1,q}(B+6d)$ that satisfies
	\begin{align*}
		\int_{R+4d}|\tau_h(\tau_hw)|^q\,dx\le M^q|h|^{q(1+\gamma)}
	\end{align*}
	for any $h\in B_d\left\lbrace 0\right\rbrace $, we have
	\begin{align*}
		\nabla w\in \left( W^{\beta,q}(B_R)\right)^N 
	\end{align*}
	for any $\beta\in(0,\gamma)$. Moreover, there exists a constant $C=C(N, q)$ such that
	\begin{align*}
		[\nabla w]^q_{W^{\beta,q}(B_R)}\le \frac{Cd^{q(\gamma-\beta)}}{(\gamma-\beta)\gamma^q(1-\gamma)^q}\left(M^q+\frac{(R+4d)^q}{\beta d^{q(1+\gamma)}}\int_{B_{R+4d}}|\nabla w|^q\,dx\right). 
	\end{align*}
\end{lemma}

\section{Local boundedness of weak solutions}
\label{sec3}
This section focuses on deriving the Caccioppoli inequality and establishing the local boundedness of weak solutions to problem \eqref{1.1}.

We first define
\begin{align*}
	u_+:=\max\left\lbrace u,0\right\rbrace,\quad u_-:=\max\left\lbrace -u,0\right\rbrace
\end{align*}
and the set
\begin{align}
	\label{L0}
	A^+(k,x_0,r):=B_{r}(x_0)\cap \left\lbrace x\in\mathbb{R}^N|u>k\right\rbrace,
\end{align}
which will be abbreviated by
\begin{align*}
	A^+(k,r):=A^+(k,x_0,r).
\end{align*}

For arbitrarily $r, t>0$, we also need to define the functions
\begin{align}
	\label{H1}
	H_r(t):=\frac{t}{r^{s_1}}+\frac{t^p}{r^{s_pp}}
\end{align}
and
\begin{align}
	\label{H2}
	h_r(t):=\frac{1}{r^{s_1}}+\frac{t^{p-1}}{r^{s_pp}}.
\end{align}
Accordingly, $H_r^{-1}$ and $h_r^{-1}$ are separately the inverses of $H_r$ and $h_r$.

For arbitrarily fixed center $x_0\in\Omega$ and radius $r\in (0,1)$ satisfying $B_{2r}\equiv B_{2r}(x_0)\subset\subset\Omega$, we define a decreasing sequence
\begin{align}
	\label{se1}
	r_i=\frac{1}{2}r+\frac{1}{2^{i+1}}r,\quad i=0,1,2,\cdots.
\end{align}
The balls $B_i$ are chosen as
\begin{align}
	\label{se2}
	B_i=B_{r_i}(x_0),\quad i=0,1,2,\cdots.
\end{align}

Now, we define the sequences of increasing level sets and the corresponding truncated functions as follows:
\begin{align*}
	k_i=\left( 1-\frac{1}{2^i}\right)k\quad\text{with}\quad k>0
\end{align*}
and
\begin{align}
	\label{se3}
	w_i=(u-k_i)_+\quad \text{for}\quad  i=0, 1, 2,\cdots.
\end{align}
Note that
\begin{align*}
	r_i-r_{i+1}=\frac{1}{2^{i+2}}r,\quad k_{i+1}-k_i=\frac{1}{2^{i+1}}k\quad\text{and} \quad w_{i+1}\le w_i,
\end{align*}
which will be used in the sequel.

 In the following two lemmas, we will address the Caccioppoli inequality formulated for the function $w_i$ over the domains $B_{i}$ and $B_{i+1}$.
\begin{lemma}
	\label{Caccioppoli}
	Let $B_{2R} \equiv B_{2R}(x_0) \subset \subset \Omega$. Suppose that $u$ is a weak solution to problem \eqref{1.1} in the sense of Definition \ref{def1}. Then, there exists a constant $C=C(N, p, s_1, s_p)$ such that for any $0<\rho<r<R$,
	\begin{align*}
		&\quad\int_{B_\rho}\int_{B_\rho}\frac{|w_+(x)-w_+(y)|}{|x-y|^{N+s_1}}\,dxdy+\int_{B_\rho}\int_{B_\rho}\frac{|w_+(x)-w_+(y)|^p}{|x-y|^{N+s_pp}}\,dxdy\\
		&\quad+\int_{B_\rho}w_+(x)\left( \int_{\mathbb{R}^N}\frac{w_-^{p-1}(y)}{|x-y|^{N+s_pp}}\,dy\right)\,dx\\
		&\le \frac{C}{r-\rho}\int_{B_r}\int_{B_r}\frac{|w_+(x)+w_+(y)|}{|x-y|^{N+s_1-1}}\,dxdy+\frac{C}{(r-\rho)^p}\int_{B_r}\int_{B_r}\frac{|w_+(x)+w_+(y)|^p}{|x-y|^{N+(s_p-1)p}}\,dxdy\\
		&\quad+\frac{Cr^{N+s_pp}}{(r-\rho)^{N+s_pp}}\int_{B_r}\int_{\mathbb{R}^N\backslash B_\rho}\frac{w_+^{p-1}(x)w_+(y)}{|x-x_0|^{N+s_pp}}\,dxdy+\int_{B_r}\int_{\mathbb{R}^N\backslash B_\rho}\frac{w_+(y)}{|x-y|^{N+s_1}}\,dxdy,
	\end{align*}
	where $w_\pm:=(u-k)_\pm$ with a level $k\in\mathbb{R}$.
\end{lemma}

\begin{proof}
	Choose any $\rho_1,r_1$ that satisfies $\rho\le\rho_1<r_1\le r$. Let $\phi\in C^\infty_0\left( B_{\frac{\rho_1+r_1}{2}}\right) $ be a cut-off function satisfying $0\le\phi\le1$, $\phi\equiv1$ in $B_{\rho_1}$ and $|\nabla\phi|\le \frac{4}{r_1-\rho_1}$. Testing $\varphi=-w_+\phi$ in problem \eqref{1.1} and using the properties of convex function $f(t)= |t|^p$, i.e.,
	\begin{align*}
		|u+t|^p-|u|^p\ge p|u|^{p-2}ut,
	\end{align*}
	one can see that
	\begin{align*}
		0&\le \int_{\mathbb{R}^N}\int_{\mathbb{R}^N}Z\frac{\varphi(x)-\varphi(y)}{|x-y|^{N+s_1}}\,dxdy\\
		&\quad+\frac{1}{p}\int_{B_{r_1}}\int_{B_{r_1}}\frac{ |(u+\varphi)(x)-(u+\varphi)(y)|^p-|u(x)-u(y)|^p }{|x-y|^{N+s_pp}}\,dxdy\\
		&\quad +\frac{2}{p}\int_{\mathbb{R}^N\backslash B_{r_1}}\int_{B_{r_1}}\frac{ |(u+\varphi)(x)-(u+\varphi)(y)|^p-|u(x)-u(y)|^p }{|x-y|^{N+s_pp}}\,dxdy\\
		&=:I_1+I_2+I_3.
	\end{align*}
	
	Next, we will estimate $I_1, I_2$ and $I_3$ respectively. Following the arguments in \cite[Proposition 7.5]{C17}, we know that
	\begin{align}
		\label{C1}
		I_2+I_3&\le -\frac{1}{C}\int_{B_{\rho_1}}\int_{B_{\rho_1}}\frac{|w_+(x)-w_+(y)|^p}{|x-y|^{N+s_pp}}\,dxdy\nonumber\\
		&\quad-\frac{1}{C}\int_{B_{\rho_1}}w_+(x)\left( \int_{\mathbb{R}^N}\frac{w_-^{p-1}(y)}{|x-y|^{N+s_pp}}\,dy\right)\,dx\nonumber\\
		&\quad+C \iint_{\left( B_{r_1}\times B_{r_1}\right)\backslash \left( B_{\rho_1}\times B_{\rho_1}\right)}\frac{|w_+(x)-w_+(y)|^p}{|x-y|^{N+s_pp}}\,dxdy\nonumber\\
		&\quad+\frac{C}{(r_1-\rho_1)^p}\int_{B_{r_1}}\int_{B_{r_1}}\frac{|w_+(x)+w_+(y)|^p}{|x-y|^{N+(s_p-1)p}}\,dxdy\nonumber\\
		&\quad+\frac{Cr^{N+s_pp}}{(r_1-\rho_1)^{N+s_pp}}\int_{B_r}\int_{\mathbb{R}^N\backslash B_\rho}\frac{w_+^{p-1}(x)w_+(y)}{|x-x_0|^{N+s_pp}}\,dxdy
	\end{align}
	with some $C\ge1$.
	
	For the term $I_1$, it is easy to see
	\begin{align}
		\label{C2}
		I_1&=-\int_{B_{r_1}}\int_{B_{r_1}}Z\frac{w_+\phi(x)-w_+\phi(y)}{|x-y|^{N+s_1}}\,dxdy-2\int_{B_{r_1}}\int_{\mathbb{R}^N\backslash B_{r_1}}Z\frac{w_+\phi(y)}{|x-y|^{N+s_1}}\,dxdy
	\end{align}
	and
	\begin{align}
		\label{3.8.1}
		\left| \int_{B_{r_1}}\int_{\mathbb{R}^N\backslash B_{r_1}}Z\frac{w_+\phi(y)}{|x-y|^{N+s_1}}\,dxdy\right| \le \int_{B_{r}}\int_{\mathbb{R}^N\backslash B_{\rho}}\frac{w_+(y)}{|x-y|^{N+s_1}}\,dxdy.
	\end{align}
	Note that
	\begin{align*}
		w_+\phi(x)-w_+\phi(y)=\left( w_+(x)-w_+(y)\right)\frac{\phi(x)+\phi(y)}{2}+\left( \phi(x)-\phi(y)\right)\frac{w_+(x)+w_+(y)}{2}, 
	\end{align*}
	we have
	\begin{align}
		\label{C3}
		\int_{B_{r_1}}\int_{B_{r_1}}Z\frac{w_+(x)-w_+(y)}{|x-y|^{N+s_1}}\frac{\phi(x)+\phi(y)}{2}\,dxdy\ge \int_{B_{\rho_1}}\int_{B_{\rho_1}}\frac{|w_+(x)-w_+(y)|}{|x-y|^{N+s_1}}\,dxdy
	\end{align}
	and
	\begin{align}
		\label{C4}
		\int_{B_{r_1}}\int_{B_{r_1}}Z\frac{\phi(x)-\phi(y)}{|x-y|^{N+s_1}}\frac{w_+(x)+w_+(y)}{2}\,dxdy\le \frac{C}{r_1-\rho_1}\int_{B_{r_1}}\int_{B_{r_1}}\frac{w_+(x)+w_+(y)}{|x-y|^{N+s_1-1}}\,dxdy.
	\end{align}
	
	We conclude from \eqref{C1}--\eqref{C4} that
	\begin{align}
		\label{C5}
		&\quad\int_{B_{\rho_1}}\int_{B_{\rho_1}}\frac{|w_+(x)-w_+(y)|}{|x-y|^{N+s_1}}\,dxdy+\int_{B_{\rho_1}}\int_{B_{\rho_1}}\frac{|w_+(x)-w_+(y)|^p}{|x-y|^{N+s_pp}}\,dxdy\nonumber\\
		&\quad+\int_{B_{\rho_1}}w_+(x)\left( \int_{\mathbb{R}^N}\frac{w_-^{p-1}(y)}{|x-y|^{N+s_pp}}\,dy\right)\,dx\nonumber\\
		&\le C \iint_{\left( B_{r_1}\times B_{r_1}\right)\backslash \left( B_{\rho_1}\times B_{\rho_1}\right)}\frac{|w_+(x)-w_+(y)|^p}{|x-y|^{N+s_pp}}\,dxdy\nonumber\\
		&\quad+\frac{C}{(r_1-\rho_1)^p}\int_{B_{r_1}}\int_{B_{r_1}}\frac{|w_+(x)+w_+(y)|^p}{|x-y|^{N+(s_p-1)p}}\,dxdy\nonumber\\
		&\quad+\frac{Cr^{N+s_pp}}{(r_1-\rho_1)^{N+s_pp}}\int_{B_r}\int_{\mathbb{R}^N\backslash B_\rho}\frac{w_+^{p-1}(x)w_+(y)}{|x-x_0|^{N+s_pp}}\,dxdy\nonumber\\
		&\quad+C\int_{B_r}\int_{\mathbb{R}^N\backslash B_\rho}\frac{w_+(y)}{|x-y|^{N+s_1}}\,dxdy+\frac{C}{r_1-\rho_1}\int_{B_{r_1}}\int_{B_{r_1}}\frac{|w_+(x)+w_+(y)|}{|x-y|^{N+s_1-1}}\,dxdy.
	\end{align}
	Now, let us define $\Phi(t)$ as
	\begin{align*}
		\Phi(t)&=\int_{B_t}\int_{B_t}\frac{|w_+(x)-w_+(y)|}{|x-y|^{N+s_1}}\,dxdy+\int_{B_t}\int_{B_t}\frac{|w_+(x)-w_+(y)|^p}{|x-y|^{N+s_pp}}\,dxdy\\
		&\quad+\int_{B_t}w_+(x)\left( \int_{\mathbb{R}^N}\frac{w_-^{p-1}(y)}{|x-y|^{N+s_pp}}\,dy\right)\,dx,\quad t>0.
	\end{align*}
	Then, it follows by \eqref{C5} that
	\begin{align*}
		\Phi(\rho_1)&\le C\left( \Phi(r_1)-\Phi(\rho_1)\right)+\frac{C}{(r_1-\rho_1)^p}\int_{B_{r}}\int_{B_{r}}\frac{|w_+(x)+w_+(y)|^p}{|x-y|^{N+(s_p-1)p}}\,dxdy\\
		&\quad+\frac{Cr^{N+s_pp}}{(r_1-\rho_1)^{N+s_pp}}\int_{B_r}\int_{\mathbb{R}^N\backslash B_\rho}\frac{w_+^{p-1}(x)w_+(y)}{|x-x_0|^{N+s_pp}}\,dxdy\\
		&\quad+\frac{C}{r_1-\rho_1}\int_{B_{r}}\int_{B_{r}}\frac{|w_+(x)+w_+(y)|}{|x-y|^{N+s_1-1}}\,dxdy+C\int_{B_r}\int_{\mathbb{R}^N\backslash B_\rho}\frac{w_+(y)}{|x-y|^{N+s_1}}\,dxdy
	\end{align*}
	with $C=C(N,p,s_1,s_p)$. This allows us to apply \cite[Lemma 2.5]{DM24} to deduce the desired estimate.
\end{proof}

\begin{lemma}
	\label{Caccioppoli2}
	Suppose that $u$ is a weak solution to problem \eqref{1.1} in the sense of Definition \ref{def1}. Let the notations $B_i$ and $w_i$ be given by \eqref{se2} and \eqref{se3}. Then, we have
	\begin{align*}
		\mint_{B_{i+1}}H_r(w_{i+1})\,dx\le \frac{C2^{(N+p)i}}{\left( H_r(k_{i+1}-k_i)\right)^\frac{1}{\kappa'}}\left( 1+\frac{\mathrm{\overline{Tail}}(u_+;x_0,\frac{r}{2})+1}{h_r(k_{i+1}-k_i)}\right)\left( \mint_{B_{i}}H_r(w_i)\,dx\right)^{1+\frac{1}{\kappa'}}, 
	\end{align*}
	where $\kappa:=\min\left\lbrace \frac{N}{N-s_1}, \frac{N}{N-s_pp}\right\rbrace$, $\kappa':=\frac{\kappa}{\kappa-1}$, and the functions $H_r(\cdot)$, $h_r(\cdot)$ are determined by \eqref{H1}, \eqref{H2} and
	\begin{align*}
		\mathrm{\overline{Tail}}\left(u_+; x_0, \frac{r}{2}\right):=\int_{\mathbb{R}^N\backslash B_{\frac{r}{2}}(x_0)}\frac{u_+^{p-1}(x)}{|x-x_0|^{N+s_pp}}\,dx.
	\end{align*}
\end{lemma}
\begin{proof}
	By the fractional Sobolev embedding theorem and H\"{o}lder's inequality, we have
	\begin{align}
		\label{L1}
		&\quad\mint_{B_{i+1}}H_r(w_{i+1})\,dx\nonumber\\
		&\le \left[ \mint_{B_{i+1}}\left( \frac{w_{i+1}}{r^{s_1}}+\frac{w_{i+1}^p}{r^{s_pp}}\right)^\kappa\,dx \right]^\frac{1}{\kappa} \left( \frac{\left| A^+(k_{i+1},r_{i+1})\right|}{|B_{i+1}|}\right) ^\frac{1}{\kappa'} \nonumber\\
		&\le C\left( \frac{\left| A^+(k_{i+1},r_{i+1})\right|}{|B_{i+1}|}\right) ^\frac{1}{\kappa'}\mint_{B_{i+1}}H_r(w_{i+1})\,dx\nonumber\\
		&\quad+C\left( \mint_{B_{i+1}}\int_{B_{i+1}}\frac{|w_{i+1}(x)-w_{i+1}(y)|}{|x-y|^{N+s_1}}\,dxdy+\mint_{B_{i+1}}\int_{B_{i+1}}\frac{|w_{i+1}(x)-w_{i+1}(y)|^p}{|x-y|^{N+s_pp}}\,dxdy\right)\nonumber\\
		&\quad\times \left( \frac{\left| A^+(k_{i+1},r_{i+1})\right|}{|B_{i+1}|}\right) ^\frac{1}{\kappa'}.
	\end{align}
	It is easy to verify that
	\begin{align*}
		\left| A^+(k_{i+1},r_{i+1})\right|\le \int_{A^+(k_{i+1},r_{i+1})}\frac{H_r(w_i)}{H_r(k_{i+1}-k_i)}\,dx\le \int_{B_i}\frac{H_r(w_i)}{H_r(k_{i+1}-k_i)},
	\end{align*}
	where the set $A^+$ is defined in \eqref{L0}. Applying Lemma \ref{Caccioppoli} and \eqref{se1}, we get
	\begin{align}
		\label{L2}
		&\quad\mint_{B_{i+1}}\int_{B_{i+1}}\frac{|w_{i+1}(x)-w_{i+1}(y)|}{|x-y|^{N+s_1}}\,dxdy+\mint_{B_{i+1}}\int_{B_{i+1}}\frac{|w_{i+1}(x)-w_{i+1}(y)|^p}{|x-y|^{N+s_pp}}\,dxdy\nonumber\\
		&\le \frac{C}{r_i-r_{i+1}}\mint_{B_i}\int_{B_i}\frac{|w_{i+1}(x)|}{|x-y|^{N+s_1-1}}\,dxdy+\frac{C}{(r_i-r_{i+1})^p}\mint_{B_i}\int_{B_i}\frac{|w_{i+1}(x)|^p}{|x-y|^{N+s_pp-p}}\,dxdy\nonumber\\
		&\quad+\frac{Cr^{N+s_pp}_i}{(r_i-r_{i+1})^{N+s_pp}}\mint_{B_i}\int_{\mathbb{R}^N\backslash B_{i+1}}\frac{w_{i+1}^{p-1}(x)w_{i+1}(y)}{|x-x_0|^{N+s_pp}}\,dxdy\nonumber\\
		&\quad+\mint_{B_i}\int_{\mathbb{R}^N\backslash B_{i+1}}\frac{w_{i+1}(y)}{|x-y|^{N+s_1}}\,dxdy\nonumber\\
		&\le \frac{Cr_i^{N+p}}{(r_i-r_{i+1})^{N+p}}\left[ \mint_{B_{i}}H_r(w_{i+1})\,dx+\left( \overline{\mathrm{Tail}}\left( u;x_0,\frac{r}{2}\right) +1\right)\mint_{B_i}w_{i+1}\,dx \right] \nonumber\\
		&\le C2^{(N+p)i}\left[ \mint_{B_i}H_r(w_{i+1})\,dx+\left( \overline{\mathrm{Tail}}\left( u;x_0,\frac{r}{2}\right) +1\right)\mint_{B_i}w_{i+1}\,dx\right].
	\end{align}
	Next, we estimate
	\begin{align*}
		\mint_{B_i}w_{i+1}\,dx\le \frac{1}{h_r(k_{i+1}-k_i)}\mint_{B_i}H_r(w_i)\,dx
	\end{align*}
	and
	\begin{align*}
		w_{i+1}\le \frac{w_i^l}{(k_{i+1}-k_i)^{l-1}}\quad\text{for}\quad l\ge1.
	\end{align*}
	Thus, bringing \eqref{L2} into \eqref{L1}, we conclude the desired estimate.
\end{proof}

\begin{proposition}
	\label{Caccioppoli3}
	Let $p>1$, $s_1,s_p\in(0,1)$ and let $u$ be a weak solution to problem \eqref{1.1} in the sense of Definition \ref{def1}. Then, $u$ is locally bounded in $\Omega$. Moreover, for any $B_{2r}\equiv B_{2r}(x_0)\subset\subset\Omega$ and any $\delta>0$, we have
	\begin{align*}
		\sup\limits_{B_{\frac{r}{2}}}u_+\le C_\delta H^{-1}_r\left( \mint_{B_r}H_r(u_+)\,dx\right)+\delta h^{-1}_r\left( r^{-s_pp}\left( \mathrm{Tail}^{p-1}\left( u;x_0,\frac{r}{2}\right) +1\right) \right),
	\end{align*}
	where the constant $C_\delta>0$ depends on $N,p,s_1,s_p$ and $\delta$.
\end{proposition}

\begin{proof}
	We define
	\begin{align*}
		Y_i:=\mint_{B_i}H_r(w_i)\,dx.
	\end{align*}
	Then, from Lemma \ref{Caccioppoli2} one can estimate $Y_i$ as
	\begin{align*}
		Y_{i+1}&\le \frac{C2^{(N+p)i}}{H_r(2^{-i-1}k)^\frac{1}{\kappa'}}\left( 1+\frac{r^{-s_pp}\mathrm{Tail}^{p-1}\left( u;x_0,\frac{r}{2}\right) +1}{h_r(2^{-i-1}k)}\right)Y_i^{1+\frac{1}{\kappa'}}\\
		&\le \frac{C2^{(N+p)i}}{H_r(2^{-i-1}k)^\frac{1}{\kappa'}}\left( 1+\frac{\mathrm{Tail}^{p-1}\left( u;x_0,\frac{r}{2}\right) +1}{r^{s_pp}h_r(2^{-i-1}k)}\right)Y_i^{1+\frac{1}{\kappa'}}.
	\end{align*}
	Note that
	\begin{align*}
		H_r(2^{-i-1}k)\ge 2^{-(i+1)p}H_r(k)
	\end{align*}
	and
	\begin{align*}
		h_r(2^{-i-1}k)\ge 2^{-(i+1)(p-1)}h_r(k),
	\end{align*}
	we get
	\begin{align*}
		Y_{i+1}&\le \frac{C2^{i(N+p)}2^{(i+1)\frac{p}{\kappa'}}2^{(i+1)(p-1)}}{\left( H_r(k)\right)^\frac{1}{\kappa'}}\left( 1+\frac{\mathrm{Tail}^{p-1}\left( u;x_0,\frac{r}{2}\right) +1}{r^{s_pp}h_r(k)}\right)Y_i^{1+\frac{1}{\kappa'}}\\
		&\le \frac{C2^{i\left( N+2p+\frac{p}{\kappa'}\right) }}{\left( H_r(k)\right)^\frac{1}{\kappa'}}\left( 1+\frac{\mathrm{Tail}^{p-1}\left( u;x_0,\frac{r}{2}\right) +1}{r^{s_pp}h_r(k)}\right)Y_i^{1+\frac{1}{\kappa'}}.
	\end{align*}
	Now, we first set $k\ge \delta h^{-1}_r\left( r^{-s_pp}\left( \mathrm{Tail}^{p-1}\left( u;x_0,\frac{r}{2}\right)+1 \right) \right) $ such that
	\begin{align}
		\label{k1}
		\frac{\delta^p\left( \mathrm{Tail}^{p-1}\left( u;x_0,\frac{r}{2}\right)+1 \right) }{r^{s_pp}h_r(k)}\le \frac{\mathrm{Tail}^{p-1}\left( u;x_0,\frac{r}{2}\right)+1}{r^{s_pp}h_r\left( \frac{k}{\delta}\right) }\le1.
	\end{align}
	Therefore, we know 
	\begin{align*}
		Y_{i+1}\le\frac{C2^{i\theta}}{\left( H_r(k)\right)^\frac{1}{\kappa'}\delta^p }Y_{i+1}^{1+\frac{1}{\kappa'}}
	\end{align*}
	with $\theta=N+2p+\frac{p}{\kappa'}$.
	
	Next, we set $\sigma:=\left( \kappa'\right)^2 $ and select $k$ large enough such that
	\begin{align*}
		Y_0\le \mint_{B_r}H_r(u_+)\,dx\le \left( \frac{2C}{\delta^p\left( H_r(k)\right)^\frac{1}{\kappa'} }\right)^{-\kappa'}2^{-\theta\sigma}.
	\end{align*}
	In other words, we select
	\begin{align}
		\label{k2}
		k\ge H^{-1}_r\left( 2^{\theta\sigma}\left( \frac{2C}{\delta^p}\right)^{\kappa'}\mint_{B_r}H_r(u_+)\,dx, \right) 
	\end{align}
	because
	\begin{align*}
		\left( \frac{2C}{\delta^p\left( H_r(k)\right)^\frac{1}{\kappa'} }\right)^{-\kappa'}2^{-\theta\sigma}=2^{-\theta\sigma}H_r(k)\left( \frac{2C}{\delta^p}\right)^{\kappa'}.
	\end{align*}
	
	To ensure the validity of \eqref{k1} and \eqref{k2}, we finally choose
	\begin{align*}
		k=H^{-1}_r\left( 2^{\theta\sigma}\left( \frac{2C}{\delta^p}\right)^{\kappa'}\mint_{B_r}H_r(u_+)\,dx \right)+\delta h^{-1}_r\left( r^{-s_pp}\left( \mathrm{Tail}^{p-1}\left( u;x_0,\frac{r}{2}\right)+1 \right) \right).
	\end{align*}
	Thus, we can apply Lemma \ref{lemY} to obtain $\lim\limits_{j\rightarrow+\infty}Y_j=0$ and the desired estimate.
\end{proof}

\begin{remark}
	\label{rem2}
	Similar to the proof of Lemmas \ref{Caccioppoli}, \ref{Caccioppoli2} and Proposition \ref{Caccioppoli3}, we can estimate the weak solution $u$ from blow, that is,
	\begin{align*}
		\sup\limits_{B_{\frac{r}{2}}}u_-\le C_\delta H^{-1}_r\left( \mint_{B_r}H_r(u_-)\,dx\right)+\delta h^{-1}_r\left( r^{-s_pp}\left( \mathrm{Tail}^{p-1}\left( u;x_0,\frac{r}{2}\right) +1\right) \right).
	\end{align*}
\end{remark}

\section{The case $s_p\in \left( 0,\frac{p-1}{p}\right] $}
\label{sec4}

We begin this section with the following Tail estimate. As we all know, for a fractional problem, the nonlocal terms are unavoidable. Therefore, when dealing with the weak formulation \eqref{1.2} via the difference quotient technique, Lemma \ref{lem8} below allows us to analyze the dependence on the step size $|h|$ and the difference $\tau_h u$. To our best knowledge, this lemma is the best result to generate the message about $|h|$. Although this advantage may be partially obscured due to the influence of the 1-structure, we will employ Lemma \ref{lem8} to address the nonlocal term of  $p$-growth. It is also worth noting that the inclusion of the 1-structure leads to the results in this paper differing from those in \cite{BDL2402}.

\begin{lemma}[Lemma 3.1, \cite{BDL2402}]
	\label{lem8}
	Let $p\ge2$ and $s_p\in (0,1)$. There exists a constant $C=\frac{\tilde{C}(N,p)}{s_p}$ such that whenever $u\in L^{p-1}_{s_pp}(\mathbb{R}^N)$, $x_0\in \mathbb{R}^N$, $R>0$, $r\in (0, R)$, and $d\in \left( 0, \frac{1}{4}(R-r)\right] $, we have, for any $x\in B_{\frac{1}{2}(R+r)}(x_0)$ and any $h\in B_d\backslash\left\lbrace 0\right\rbrace$ with $0<|h|\le d$ that
	\begin{align*}
		&\quad\left| \int_{\mathbb{R}^N\backslash B_R(x_0)}\frac{J_p(u_h(x)-u_h(y))-J_p(u(x)-u(y))}{|x-y|^{N+s_pp}}dy\right|\\
		&\le C\frac{|\tau_hu(x)|}{R^{s_pp}}\left( \frac{R}{R-r}\right)^{N+s_pp}\mathcal{T}^{p-2}+C\frac{|h|}{R^{s_pp+1}}\left( \frac{R}{R-r}\right)^{N+s_pp+1}\mathcal{T}^{p-1},  
	\end{align*}
	where
	\begin{align*}
		\mathcal{T}:=\|u\|_{L^\infty(B_{R+d}(x_0))}+\mathrm{Tail}(u;x_0,R+d).
	\end{align*}
\end{lemma}

Next, we are going to derive a local estimate. Some of the reasoning is based on the ideas developed in \cite[Proposition 3.3]{BDL2402}. Thanks to the $p$-growth, one can expect a positive term on the left-hand side of the energy inequality. This inequality allows us to obtain
\begin{align*}
	\|\tau_h(\tau_hu)\|^q_{L^q(B_r)}\le C\left( |h|^{s_p p+\sigma(q-p+1)}+|h|^{s_p p+\sigma(q-p+2)}\right).
\end{align*}
The term $|h|^{s_pp+\sigma(q-p+1)}$ comes from 1-growth since we can not expect any regularity result for $Z$. Thus, $p$-growth is controlled by 1-growth.

From now on, we define
\begin{align*}
	\mu^{s_p}_p:=
	\begin{cases}
		\frac{s_pp-2}{p-2}\quad &\text{if }p>2,\\
		0\quad &\text{if }p=2.
	\end{cases}
\end{align*}
\begin{proposition}
	\label{pro9}
	Let $p\ge2$, $s_1,s_p\in (0,1)$, $q\ge p$ and $\sigma\in \left( \max\left\lbrace 0,\mu^{s_p}_p\right\rbrace, \min\left\lbrace 1,\frac{s_pp}{p-1}\right\rbrace  \right) $. 
	There exists a constant $C=C(N,p,q,s_1,s_p,\sigma)$ such that for any $u\in W^{\sigma,q}_{\mathrm{loc}}(\Omega)$ being a locally bounded weak solution to problem \eqref{1.1} in the sense of Definition \ref{def1}, we have, for any $r\in(0,R)$ with $R\in (0,1)$, $d\in\left( 0,\frac{1}{4}(R-r)\right] $, $B_{R+d}\equiv B_{R+d}(x_0)\subset\subset\Omega$, $h\in B_d\backslash\left\lbrace 0\right\rbrace $,
	\begin{align*}
		&\quad\int_{B_r}\int_{B_r}\frac{\left| J_{\frac{q}{p}+1}(\tau_hu)(x)-J_{\frac{q}{p}+1}(\tau_hu)(y) \right|^p }{|x-y|^{N+s_pp}}\,dxdy\\
		&\le C\frac{R^{1-s_1}}{R-r}\int_{B_R}|\tau_hu|^{q-p+1}\,dx+C\frac{1}{s_1(R-r)^{s_1}}\int_{B_R}|\tau_hu|^{q-p+1}\,dx\\
		&\quad+C\frac{R^{(N-\varepsilon)\frac{q-p+2}{q}}}{(R-r)^2}[u]^{p-2}_{W^{\sigma,q}(B_{R+d})}\left( \int_{B_R}|\tau_hu|^{q}\,dx\right)^{\frac{q-p+2}{q}}\\
		&\quad+C\frac{\mathcal{T}^{p-2}}{(R-r)^{N+s_pp}}\int_{B_R}|\tau_hu|^{q-p+2}\,dx+C\frac{\mathcal{T}^{p-1}|h|}{(R-r)^{N+s_pp+1}}\int_{B_R}|\tau_hu|^{q-p+1}\,dx,
	\end{align*}
	where
	\begin{align*}
		\varepsilon=\left( N+s_pp-2-\frac{(N+\sigma q)(p-2)}{q}\right)\frac{q}{q-p+2}
	\end{align*}
	and
		\begin{align*}
			\mathcal{T}:=\|u\|_{L^\infty(B_{R+d}(x_0))}+\mathrm{Tail}(u;x_0,R+d).
		\end{align*}
\end{proposition}

\begin{proof}
	For a ball $B_{R+d}\equiv B_{R+d}(x_0)\subset\subset\Omega$, $0<r<R\le1$ and $d\in\left( 0,\frac{1}{4}(R-r)\right]$, by testing \eqref{1.2} with $\varphi_{-h}:=\varphi(x-h)$, where $\varphi\in W^{s_1,1}(B_R)\cap W^{s_p, p}(B_R)$ with $\mathrm{supp}\,\varphi\subset B_{\frac{1}{2}(R+r)}$, we conclude
	\begin{align}
		\label{9.1}
		0&=\int_{\mathbb{R}^N}\int_{\mathbb{R}^N}Z_h\frac{\varphi(x)-\varphi(y)}{|x-y|^{N+s_1}}\,dxdy\nonumber\\
		&\quad+\int_{\mathbb{R}^N}\int_{\mathbb{R}^N}\frac{|u_h(x)-u_h(y)|^{p-2}(u_h(x)-u_h(y))(\varphi(x)-\varphi(y))}{|x-y|^{N+s_pp}}\,dxdy.
	\end{align}
	Subtracting \eqref{1.2} from \eqref{9.1}, we obtain
	\begin{align}
		\label{9.2}
		0&=\int_{\mathbb{R}^N}\int_{\mathbb{R}^N}(Z_h-Z)\frac{\varphi(x)-\varphi(y)}{|x-y|^{N+s_1}}\,dxdy\nonumber\\
		&\quad+\int_{\mathbb{R}^N}\int_{\mathbb{R}^N}\frac{\left[ J_p(u_h(x)-u_h(y))-J_p(u(x)-u(y))\right] (\varphi(x)-\varphi(y))}{|x-y|^{N+s_pp}}\,dxdy
	\end{align}
	for any $\varphi\in W^{s_1,1}(B_R)\cap W^{s_p,p}(B_R)$ with $\mathrm{supp}\,\varphi\subset B_{\frac{1}{2}(R+r)}$.
	
	Now, we choose
	\begin{align*}
		\varphi:=J_{q-p+2}(\tau_hu)\eta^p
	\end{align*}
	with $\eta\in C^1_0(B_{\frac{1}{2}(R+r)})$, $|\nabla \eta|\le \frac{C}{R-r}$ and $\eta\equiv1$ in $B_r$ in \eqref{9.1}. One can easily verify that $\varphi\in W^{s_1,1}(B_R)\cap W^{s_p, p}(B_R)$,  because $u$ is locally bounded in $\Omega$. Thus, we conclude
	\begin{align}
		\label{9.3}
		0&=\int_{\mathbb{R}^N}\int_{\mathbb{R}^N}(Z_h-Z)\frac{J_{q-p+2}(\tau_hu)\eta^p(x)-J_{q-p+2}(\tau_hu)\eta^p(y)}{|x-y|^{N+s_1}}\,dxdy\nonumber\\
		&\quad+\int_{\mathbb{R}^N}\int_{\mathbb{R}^N}\frac{J_p(u_h(x)-u_h(y))-J_p(u(x)-u(y))}{|x-y|^{N+s_pp}}\nonumber\\
		&\qquad\qquad\qquad\times\left[ J_{q-p+2}(\tau_hu)\eta^p(x)-J_{q-p+2}(\tau_hu)\eta^p(y)\right]\,dxdy\nonumber\\
		&=:I+P.
	\end{align}
	
	Next, we will estimate the terms $I$ and $P$, respectively. For the term $I$, we have the following decomposition
	\begin{align*}
		I&=\int_{B_R}\int_{B_R}(Z_h-Z)\frac{J_{q-p+2}(\tau_hu)\eta^p(x)-J_{q-p+2}(\tau_hu)\eta^p(y)}{|x-y|^{N+s_1}}\,dxdy\\
		&\quad+2\int_{B_{\frac{1}{2}(R+r)}}\int_{\mathbb{R}^N\backslash B_R}(Z_h-Z)\frac{J_{q-p+2}(\tau_hu)\eta^p(x)}{|x-y|^{N+s_1}}\,dxdy\\
		&=:I_1+2I_2.
	\end{align*}
	We now pay our attention to the local term $I_1$ to deduce
	\begin{align*}
		I_1&=\int_{B_R}\int_{B_R}(Z_h-Z)\frac{J_{q-p+2}(\tau_hu)(x)-J_{q-p+2}(\tau_hu)(y)}{|x-y|^{N+s_1}}\frac{\eta^p(x)+\eta^p(y)}{2}\,dxdy\\
		&\quad+\int_{B_R}\int_{B_R}(Z_h-Z)\frac{\eta^p(x)-\eta^p(y)}{|x-y|^{N+s_1}}\frac{J_{q-p+2}(\tau_hu)(x)+J_{q-p+2}(\tau_hu)(y)}{2}\,dxdy\\
		&=:I_{1,1}+I_{1,2}.
	\end{align*}
	
	By the definitions of $Z$ and $Z_h$, we have, for $(x,y)\in B_R\times B_R$,
	\begin{align*}
		Z_h-Z=0
	\end{align*}
	in  the domains $\left\lbrace u_h(x)>u_h(y)\text{ and }u(x)>u(y)\right\rbrace $ and $\left\lbrace u_h(x)<u_h(y)\text{ and }u(x)<u(y)\right\rbrace $.
	
	In the domain $\left\lbrace u_h(x)<u_h(y)\text{ and }u(x)\ge u(y)\right\rbrace $, we have $\tau_hu(x)\le \tau_hu(y)$, which follows that
		\begin{align*}
		J_{q-p+2}(\tau_hu)(x)\le J_{q-p+2}(\tau_hu)(y)
	\end{align*}
	and
	\begin{align}
		\label{9.4}
		(Z_h-Z)\left[ J_{q-p+2}(\tau_hu)(x)-J_{q-p+2}(\tau_hu)(y)\right]\ge 0.
	\end{align}
Similarly, we know that \eqref{9.4} holds in the following four domains
	\begin{align*}
		\left\lbrace u_h(x)\ge u_h(y)\text{ and }u(x)<u(y)\right\rbrace, \quad \left\lbrace u_h(x)>u_h(y)\text{ and }u(x)=u(y)\right\rbrace,\\
		\left\lbrace u_h(x)=u_h(y)\text{ and }u(x)> u(y)\right\rbrace,  \quad	\left\lbrace u_h(x)=u_h(y)\text{ and }u(x)=u(y)\right\rbrace.
	\end{align*}
	Thus, we obtain the non-negativity of $I_{1,1}$.
	
	We now estimate $I_{1,2}$. For the cut-off functions, we have
	\begin{align*}
		|\eta^p(x)-\eta^p(y)|\le p\|\nabla \eta\|_{L^{\infty}(B_R)}|x-y|.
	\end{align*}
	Hence,
	\begin{align}
		\label{9.5}
		I_{1,2}&\le \int_{B_R}\int_{B_R}\frac{C(p)}{R-r}\frac{|J_{q-p+2}(\tau_hu)(x)|}{|x-y|^{N+s_1-1}}\,dxdy\nonumber\\
		&\le \frac{C(N,p)}{R-r}\int_0^{2R}\frac{1}{\rho^{s_1}}\,d\rho\int_{B_R}|\tau_hu|^{q-p+1}\,dx\nonumber\\
		&\le \frac{C(N,p)}{(1-s_1)(R-r)}\int_{B_R}|\tau_hu|^{q-p+1}\,dx.
	\end{align}
	As to the nonlocal term $I_2$, we have
	\begin{align}
		\label{9.6}
		I_2&\le C(N)\int_{\frac{R-r}{2}}^{+\infty}\frac{1}{\rho^{s_1+1}}\,d\rho\int_{B_{\frac{1}{2}(R+r)}}|\tau_hu|^{q-p+1}\,dx\nonumber\\
		&\le \frac{C(N)}{s_1}\frac{1}{(R-r)^{s_1}}\int_{B_{\frac{1}{2}(R+r)}}|\tau_hu|^{q-p+1}\,dx.
	\end{align}
	
	For the term  $P$, note that
	\begin{align*}
		P&=\int_{B_R}\int_{B_R}\frac{J_p(u_h(x)-u_h(y))-J_p(u(x)-u(y))}{|x-y|^{N+s_pp}}\\
		&\qquad\qquad\qquad\times\left[ J_{q-p+2}(\tau_hu)\eta^p(x)-J_{q-p+2}(\tau_hu)\eta^p(y)\right]\,dxdy\\
		&\quad+2\int_{B_{\frac{1}{2}(R+r)}}\int_{\mathbb{R}^N\backslash B_R}\frac{J_p(u_h(x)-u_h(y))-J_p(u(x)-u(y))}{|x-y|^{N+s_pp}}J_{q-p+2}(\tau_hu)\eta^p(x)\,dxdy\\
		&=:P_1+2P_2.
	\end{align*}
	We are going to estimate $P_1$ and $P_2$. By applying Lemma \ref{lem3}, we have
	\begin{align*}
		&\quad\left[ J_p(u_h(x)-u_h(y))-J_p(u(x)-u(y))\right]\left[ J_{q-p+2}(\tau_hu)\eta^p(x)-J_{q-p+2}(\tau_hu)\eta^p(y)\right]\\
		&\ge\frac{1}{C}\left( |u_h(x)-u_h(y)|+|u(x)-u(y)|\right)^{p-2}\left( |\tau_hu(x)|+|\tau_hu(y)|\right)^{q-p}\\
		&\qquad\times|\tau_hu(x)-\tau_hu(y)|^2\left( \eta^p(x)+\eta^p(y)\right)\\
		&\quad-C\left( |\tau_hu(x)|+|\tau_hu(y)|\right)^{q-p+2}\left( |u_h(x)-u_h(y)|+|u(x)-u(y)|\right)^{p-2}|\eta^{\frac{p}{2}}(x)-\eta^{\frac{p}{2}}(y)|^2\\
		&\ge \frac{1}{C}\left( |\tau_hu(x)|+|\tau_hu(y)|\right)^p\left( |\tau_hu(x)|+|\tau_hu(y)|\right)^{q-p} \left( \eta^p(x)+\eta^p(y)\right)\\
		&\quad -C\left( |\tau_hu(x)|+|\tau_hu(y)|\right)^{q-p+2}\left( |u_h(x)-u_h(y)|+|u(x)-u(y)|\right)^{p-2}|\eta^{\frac{p}{2}}(x)-\eta^{\frac{p}{2}}(y)|^2.
	\end{align*}
	Thus, for the local term $P_1$, from Lemma \ref{lem2}, we obtain
	\begin{align}
		\label{9.6.1}
		P_1&\ge \frac{1}{C}\int_{B_R}\int_{B_R}\frac{\left|J_{\frac{q}{p}+1}(\tau_hu)(x)-J_{\frac{q}{p}+1}(\tau_hu)(y)\right|^p(\eta^p(x)+\eta^p(y))}{|x-y|^{N+s_pp}}\,dxdy\nonumber\\
		&\quad -C\int_{B_R}\int_{B_R}\frac{\left( |\tau_hu(x)|+|\tau_hu(y)|\right)^{q-p+2}\left( |u_h(x)-u_h(y)|+|u(x)-u(y)|\right)^{p-2}}{|x-y|^{N+s_pp}}\nonumber\\
		&\qquad\qquad\qquad\quad\times|\eta^{\frac{p}{2}}(x)-\eta^{\frac{p}{2}}(y)|^2\,dxdy\nonumber\\
		&=:\frac{1}{C(p,q)}P_{1,1}-C(p,q)P_{1,2}.
	\end{align}
	Recalling that
	\begin{align*}
		|\eta^{\frac{p}{2}}(x)-\eta^{\frac{p}{2}}(y)|\le \frac{C(p)}{R-r}|x-y|,
	\end{align*}
	one can estimate $P_{1,2}$ as 
	\begin{align}
		\label{9.7}
		|P_{1,2}|&\le \frac{C(p)}{(R-r)^2}\int_{B_R}\int_{B_R}\frac{\left( |\tau_hu(x)|+|\tau_hu(y)|\right)^{q-p+2}}{|x-y|^{N+s_pp-2}}\nonumber\\
		&\qquad\qquad\qquad\qquad\quad\times\left( |u_h(x)-u_h(y)|+|u(x)-u(y)|\right)^{p-2}\,dxdy\nonumber\\
		&\le \frac{C(p,q)}{(R-r)^2}\left[ \int_{B_{R+d}}\int_{B_{R+d}}\frac{|u(x)-u(y)|^q}{|x-y|^{N+\sigma q}}\,dxdy\right]^{\frac{p-2}{q}}\left[ \int_{B_{R}}\int_{B_{R}}\frac{|\tau_hu(x)|^q}{|x-y|^{\varepsilon}}\,dxdy\right]^\frac{q-p+2}{q},  
	\end{align}
	where H\"{o}lder's inequality with exponents $\frac{q}{p-2}$ and $\frac{q}{q-p+2}$ was utilized in the last line and 
	\begin{align*}
		\varepsilon=\left(N+s_pp-2-\frac{(N+\sigma q)(p-2)}{q}\right)\frac{q}{q-p+2}. 
	\end{align*}
	
	For the nonlocal term $P_2$, we make use of Lemma \ref{lem8} to deduce
	\begin{align}
		\label{9.8}
		|P_2|&\le \frac{C}{R^{s_pp}}\left( \frac{R}{R-r}\right)^{N+s_pp}\mathcal{T}^{p-2}\int_{B_R}|\tau_hu|^{q-p+2}\,dx\nonumber\\
		&\quad +\frac{C|h|}{R^{s_pp+1}}\left( \frac{R}{R-r}\right)^{N+s_pp+1}\mathcal{T}^{p-1}\int_{B_R}|\tau_hu|^{q-p+1}\,dx.
	\end{align}
	By combining \eqref{9.3}, \eqref{9.5}--\eqref{9.8} and taking $I_{1,1}\ge 0$ into consideration, we conclude
	\begin{align*}
		&\quad\int_{B_r}\int_{B_r}\frac{\left| J_{\frac{q}{p}+1}(\tau_hu)(x)-J_{\frac{q}{p}+1}(\tau_hu)(y) \right|^p }{|x-y|^{N+s_pp}}\,dxdy\\
		&\le C\frac{R^{1-s_1}}{R-r}\int_{B_R}|\tau_hu|^{q-p+1}\,dx+C\frac{1}{s_1(R-r)^{s_1}}\int_{B_R}|\tau_hu|^{q-p+1}\,dx\\
		&\quad+C\frac{R^{(N-\varepsilon)\frac{q-p+2}{q}}}{(R-r)^2}[u]^{p-2}_{W^{\sigma,q}(B_{R+d})}\left( \int_{B_R}|\tau_hu|^{q}\,dx\right)^{\frac{q-p+2}{q}}\\
		&\quad+C\frac{\mathcal{T}^{p-2}}{(R-r)^{N+s_pp}}\int_{B_R}|\tau_hu|^{q-p+2}\,dx+C\frac{\mathcal{T}^{p-1}|h|}{(R-r)^{N+s_pp+1}}\int_{B_R}|\tau_hu|^{q-p+1}\,dx.
	\end{align*}
\end{proof}

From now on, we begin our proof of the main results with small $s_p$. Our argument contains two steps. The first step is to show that for a fixed $q\ge p$, one can raise the fractional differentiability index to any number less than $\frac{s_pp}{p-1}$; the second step is to raise the integrability order from the given $p$ to any $q\ge p$ and the differentiability will be sacrificed.

\begin{lemma}
	\label{lem10}
	Let $p\ge2$, $s_p\in \left( 0,\frac{p-1}{p}\right]$ and $q\ge p$. Then, whenever $u\in W^{\sigma,q}_{\mathrm{loc}}(\Omega)$ is locally bounded weak solution to problem \eqref{1.1} in the sense of Definition \ref{def1} with
	\begin{align*}
		\sigma\in \left( \max\left\lbrace \mu^{s_p}_p,0\right\rbrace, \frac{s_pp}{p-1} \right),
	\end{align*}
	we have
	\begin{align*}
		u\in W^{\alpha,q}_{\mathrm{loc}}(\Omega)
	\end{align*}
	for any $\alpha\in (\sigma,\beta)$, where
	\begin{align*}
		\beta:=\left( 1-\frac{p-1}{q}\right)\sigma+\frac{s_pp}{q}.
	\end{align*}
	Moreover, there exists a constant $C=C(N, p, q, s_1, s_p, \sigma, \alpha)$ such that for any ball $B_R\equiv B_R(x_0)\subset\subset\Omega$ with $R\in (0,1)$ and any $r\in (0,R)$,
	\begin{align*}
		[u]^q_{W^{\alpha,q}(B_r)}\le \frac{C}{(R-r)^{N+2q+2}}\left( \mathcal{T}+[u]_{W^{\sigma,q}(B_R)}+1\right)^q.
	\end{align*}
\end{lemma}

\begin{proof}
	To begin with, we will apply the energy estimate in Proposition \ref{pro9}. We define $d:=\frac{1}{7}(R-r)$ and replace $r,R$ by $\tilde{r}:=r+d$ and $\tilde{R}=:R-d$ to obtain
	\begin{align*}
		&\quad\int_{B_{\tilde{r}}}\int_{B_{\tilde{r}}}\frac{\left| J_{\frac{q}{p}+1}(\tau_hu)(x)-J_{\frac{q}{p}+1}(\tau_hu)(y)\right|^p}{|x-y|^{N+s_pp}}\,dxdy\\
		&\le \frac{C}{(\tilde{R}-\tilde{r})}\int_{B_{\tilde{R}}}|\tau_hu|^{q-p+1}\,dx+ \frac{C}{(\tilde{R}-\tilde{r})^{s_1}}\int_{B_{\tilde{R}}}|\tau_hu|^{q-p+1}\,dx\\
		&\quad+\frac{C}{(\tilde{R}-\tilde{r})^{2}}[u]^{p-2}_{W^{\sigma,q}(B_{\tilde{R}})}\left( \int_{B_{\tilde{R}}}|\tau_hu|^{q}\,dx\right)^{\frac{q-p+2}{q}}\\
		&\quad+\frac{C\mathcal{T}^{p-2}}{(\tilde{R}-\tilde{r})^{N+s_pp}}\int_{B_{\tilde{R}}}|\tau_hu|^{q-p+2}\,dx+\frac{C\mathcal{T}^{p-1}|h|}{(\tilde{R}-\tilde{r})^{N+s_pp+1}}\int_{B_{\tilde{R}}}|\tau_hu|^{q-p+1}\,dx,
	\end{align*}
	where we make use of Lemma \ref{lem4} to estimate
	\begin{align*}
		\|u\|_{L^\infty(B_{\tilde{R}})}+\mathrm{Tail}(u;x_0,\tilde{R})&\le C(N,p)\left( \frac{R}{\tilde{R}}\right)^N\left( \|u\|_{L^\infty(B_R)}+\mathrm{Tail}(u;x_0,R)\right)\\
		&\le C(N,p)\left( \frac{7}{6}\right)^N\left( \|u\|_{L^\infty(B_R)}+\mathrm{Tail}(u;x_0,R)\right).
	\end{align*}
	In view of Lemma \ref{lem5} and $u\in W^{\sigma,q}(B_R)$, we have
	\begin{align}
		\label{10.1}
		\int_{B_{\tilde{R}}}|\tau_hu|^q\,dx\le C|h|^{\sigma q}\left( (1-\sigma)[u]^q_{W^{\sigma,q}(B_{\tilde{R}+d})}+\frac{7R^N}{\sigma(R-r)^q}\|u\|_{L^{\infty}(B_{\tilde{R}})}\right). 
	\end{align}
	Therefore, from H\"{o}lder's inequality, we obtain
	\begin{align}
		\label{10.2}
		&\quad\int_{B_{\tilde{r}}}\int_{B_{\tilde{r}}}\frac{\left| J_{\frac{q}{p}+1}(\tau_hu)(x)-J_{\frac{q}{p}+1}(\tau_hu)(y)\right|^p}{|x-y|^{N+s_pp}}\,dxdy\nonumber\\
		&\le \frac{C}{(\tilde{R}-\tilde{r})}\left( \int_{B_{\tilde{R}}}|\tau_hu|^{q}\,dx\right)^{\frac{q-p+1}{q}}+\frac{C}{(\tilde{R}-\tilde{r})^{s_1}}\left( \int_{B_{\tilde{R}}}|\tau_hu|^{q}\,dx\right)^{\frac{q-p+1}{q}}\nonumber\\
		&\quad+\frac{C\mathcal{T}^{p-1}|h|}{(\tilde{R}-\tilde{r})^{N+s_pp+1}}\left( \int_{B_{\tilde{R}}}|\tau_hu|^{q}\,dx\right)^{\frac{q-p+1}{q}}\nonumber\\
		&\quad+\frac{C}{(\tilde{R}-\tilde{r})^{2}}[u]^{p-2}_{W^{\sigma,q}(B_{\tilde{R}})}\left( \int_{B_{\tilde{R}}}|\tau_hu|^{q}\,dx\right)^{\frac{q-p+2}{q}}+\frac{C\mathcal{T}^{p-2}}{(\tilde{R}-\tilde{r})^{N+s_pp}}\left( \int_{B_{\tilde{R}}}|\tau_hu|^{q}\,dx\right)^{\frac{q-p+2}{q}}.
	\end{align}
	For the left-hand side of \eqref{10.2}, Lemma \ref{lem5} implies, for any $\lambda\in B_d\backslash\left\lbrace 0\right\rbrace $,
	\begin{align}
		\label{10.3}
		\int_{B_r}\left| \tau_\lambda\left( J_{\frac{q}{p}+1}(\tau_hu)\right) \right|^p\,dx&\le C(1-s_p)|h|^{s_pp}\left[ J_{\frac{q}{p}+1}(\tau_hu)\right]^p_{W^{s_p,p}(B_{\tilde{r}})}\nonumber\\
		&\le\frac{C|h|^{s_pp}}{s_p(\tilde{R}-\tilde{r})^p}\int_{B_{\tilde{r}}}|\tau_hu|^q\,dx .
	\end{align}
	From \eqref{10.1}--\eqref{10.3}, we conclude
	\begin{align}
		\label{10.4}
		&\quad\int_{B_r}\left| \tau_\lambda\left( J_{\frac{q}{p}+1}(\tau_hu)\right) \right|^p\,dx\nonumber\\
		&\le \frac{C|\lambda|^{s_pp}\left( \mathcal{T}+1\right) ^{p-1}}{(\tilde{R}-\tilde{r})^{N+s_pp+1}}\left( \int_{B_{\tilde{R}}}|\tau_hu|^{q}\,dx\right)^{\frac{q-p+1}{q}}\nonumber\\
		&\quad+\frac{C|\lambda|^{s_pp}\left( \mathcal{T}+[u]_{W^{\sigma,q}(B_{\tilde{R}})}\right) ^{p-1}}{(\tilde{R}-\tilde{r})^{N+s_pp}}\left( \int_{B_{\tilde{R}}}|\tau_hu|^{q}\,dx\right)^{\frac{q-p+2}{q}}+\frac{C|\lambda|^{s_pp}}{s_p(\tilde{R}-\tilde{r})^{p}}\int_{B_{\tilde{r}}}|\tau_hu|^{q}\,dx\nonumber\\
		&\le \frac{C|\lambda|^{s_pp}\left( \mathcal{T}+1\right)^{p-1} }{(\tilde{R}-\tilde{r})^{N+s_pp+1}}\left[ |h|^{\sigma q}\left( [u]^q_{W^{\sigma,q}(B_{\tilde{R}+d})}+\frac{1}{\sigma(R-r)^q}\|u\|_{L^\infty(B_{\tilde{R}})} \right)\right]^{\frac{q-p+1}{q}} \nonumber\\
		&\quad+\frac{C|\lambda|^{s_pp}\left( \mathcal{T}+[u]_{W^{\sigma,q}(B_{\tilde{R}})}\right)^{p-1} }{(\tilde{R}-\tilde{r})^{N+s_pp}}\left[ |h|^{\sigma q}\left( [u]^q_{W^{\sigma,q}(B_{\tilde{R}+d})}+\frac{1}{\sigma(R-r)^q}\|u\|_{L^\infty(B_{\tilde{R}})} \right) \right]^{\frac{q-p+2}{q}} \nonumber\\
		&\quad+\frac{C|\lambda|^{s_pp}}{s_p(\tilde{R}-\tilde{r})^{p}}\left[  |h|^{\sigma q}\left( [u]^q_{W^{\sigma,q}(B_{\tilde{R}+d})}+\frac{1}{\sigma(R-r)^q}\|u\|_{L^\infty(B_{\tilde{R}})} \right) \right].
	\end{align}
	By choosing $\lambda=h$ and observing that
	\begin{align*}
		C(p,q)\left| \tau_h\left( J_{\frac{q}{p}+1}(\tau_hu)\right) \right|\ge \left| \tau_h\left( \tau_hu\right) \right|^{\frac{q}{p}} , 
	\end{align*}
	the left-hand side of \eqref{10.4} can be estimated from blow as 
	\begin{align*}
		\int_{B_r}\left| \tau_h\left( \tau_hu\right) \right|^q\,dx\le C(p,q)\int_{B_r}\left| \tau_h\left( J_{\frac{q}{p}+1}(\tau_hu)\right) \right|^p\,dx.
	\end{align*}
	Therefore, note that $|h|<R<1$, from \eqref{10.4} we have
	\begin{align*}
		\int_{B_r}\left| \tau_h\left( \tau_hu\right) \right|^q\,dx\le \frac{C\left( \mathcal{T}+[u]_{W^{\sigma,q}(B_R)}+1\right)^q }{(R-r)^{N+2q+2}}|h|^{s_pp+\sigma(q-p+1)}
	\end{align*}
	for any $h\in B_d\backslash\left\lbrace 0\right\rbrace $. Now, we apply Lemma \ref{lem6} to obtain
	\begin{align*}
		\int_{B_r}|\tau_hu|^q\,dx&\le C\left[ \frac{\left( \mathcal{T}+[u]_{W^{\sigma,q}(B_R)}+1\right)^q }{(R-r)^{N+2q+2}}+\frac{\|u\|^q_{L^\infty(B_R)}}{(R-r)^{s_pp+\sigma(q-p+1)}}\right] |h|^{s_pp+\sigma(q-p+1)}\\
		&\le \frac{C\left( \mathcal{T}+[u]_{W^{\sigma,q}(B_R)}+1\right)^q}{(R-r)^{s_pp+\sigma(q-p+1)}}|h|^{s_pp+\sigma(q-p+1)},
	\end{align*}
	for any $h\in B_{\frac{1}{2}d}\backslash\left\lbrace 0\right\rbrace $. As a result of application of Lemma \ref{lem7}, we conclude
	\begin{align*}
		\int_{B_r}\int_{B_r}\frac{|u(x)-u(y)|^q}{|x-y|^{N+\alpha q}}\,dxdy&\le C\left[ \frac{\left( \mathcal{T}+[u]_{W^{\sigma,q}(B_R)}+1\right)^q }{(\beta-\alpha)(R-r)^{N+2q+2}}+\frac{\|u\|^q_{L^\infty(B_R)}}{\alpha(R-r)^{\alpha q}}\right]\\
		&\le \frac{C}{(R-r)^{N+2q+2}}\left( \mathcal{T}+[u]_{W^{\sigma,q}(B_R)}+1\right)^q.
	\end{align*}
\end{proof}

\begin{remark}
	\label{remC}
	After the calculation, we know that the constant $C$ takes the form
	\begin{align*}
		C=\frac{\tilde{C}(N,p,q)}{s_1s_p(1-s_1)(N-\varepsilon)^\frac{q-p+2}{q}\sigma\alpha(\beta-\alpha)}
	\end{align*}
	with
	\begin{align*}
		\varepsilon=\left(N+s_pp-2-\frac{(N+\sigma q)(p-2)}{q}\right)\frac{q}{q-p+2}.
	\end{align*}
\end{remark}

In the following lemma, we iterate the improvement in fractional differentiability obtained in Lemma \ref{lem10}.
\begin{lemma}
	\label{lem11}
	Let $p\ge2$, $s_p\in\left( 0, \frac{p-1}{p}\right] $, $q\ge p$ and $\sigma\in \left( \max\left\lbrace \mu^{s_p}_p, 0\right\rbrace, \frac{s_pp}{p-1} \right) $. Suppose that $u\in W^{\sigma,q}_\mathrm{loc}(\Omega)$ is a locally bounded weak solution to problem \eqref{1.1} in the sense of Definition \ref{def1}. Then, we have
	\begin{align*}
		u\in W^{\gamma,q}_\mathrm{loc}(\Omega)
	\end{align*}
	for any $\gamma\in \left( \sigma,\frac{s_pp}{p-1}\right) $. Moreover, there exist two constants $C$ and $\kappa$ depending on $N, p, q, s_1, s_p$ and $ \gamma$, such that for any $r\in (0,R)$ with $R\in (0,1]$ and any ball $B_R\equiv B_R(x_0)\subset\subset \Omega$,
	\begin{align*}
		[u]^q_{W^{\gamma,q}(B_r)}\le \frac{C\left( \mathcal{T}+[u]_{W^{\sigma,q}(B_R)}+1\right)^q }{(R-r)^\kappa}.
	\end{align*}
\end{lemma}

\begin{proof}
	We set $\tilde{\gamma}:=\frac{1}{2}\left( \gamma+\frac{s_pp}{p-1}\right)\in \left( \gamma,\frac{s_pp}{p-1}\right)$ and define
	\begin{align*}
		\sigma_0:=\sigma, \quad \sigma_{i+1}:=\left( 1-\frac{p-1}{q}\right)\sigma_i+\frac{\tilde{\gamma}(p-1)}{q} 
	\end{align*}
	for $i\in \mathbb{N}$. It is easy to see that $\sigma_i\rightarrow\tilde{\gamma}$ as $i\rightarrow+\infty$. Moreover, we define the sequence
	\begin{align*}
		\rho_i:=r+\frac{1}{2^i}(R-r)
	\end{align*}
	and
	\begin{align*}
		\beta_i:=\left( 1-\frac{p-1}{q}\right)\sigma_{i-1}+\frac{s_pp}{q}.
	\end{align*}
	Thus, we have
	\begin{align*}
		R>\rho_{i-1}=\frac{R}{2^{i-1}}+r\left( 1-\frac{1}{2^{i-1}}\right)>\frac{R}{2^{i-1}} 
	\end{align*}
	and
	\begin{align}
		\label{11.1}
		\frac{1}{\rho_{i-1}-\rho_{i}}<\frac{2^i}{R-r}.
	\end{align}
	
	To express the estimate exactly, we also define
	\begin{align*}
		\mathcal{T}_i:=\|u\|_{L^\infty(B_{\rho_i})}+\mathrm{Tail}(u;x_0,\rho_i).
	\end{align*}
	Lemma \ref{lem4} implies
	\begin{align}
		\label{11.2}
		\mathrm{Tail}(u;x_0,\rho_{i-1})^{p-1}\le C(N)\left( \frac{R}{\rho_{i-1}}\right)^N \mathcal{T}^{p-1}\le C(N)2^{iN}\mathcal{T}^{p-1}.
	\end{align}
	
	Now, we apply Lemma \ref{lem10} with parameters $(\alpha,\beta,\sigma,r,R)$ replaced by $(\sigma_i,\beta_i,\sigma_{i-1},\rho_i,\rho_{i-1})$ to obtain
	\begin{align}
		\label{11.2-1}
		[u]^q_{W^{\sigma_i,q}(B_{\rho_i})}\le \frac{C\left( \mathcal{T}_{i-1}+[u]_{W^{\sigma_{i-1},q}(B_{\rho_{i-1}})}+1\right)^q}{\left( \rho_{i-1}-\rho_i\right)^{N+2q+2} }.
	\end{align}
	This application is permitted as long as $[u]_{W^{\sigma_{i-1},q}(B_{\rho_{i-1}})}<+\infty$. Next, we are going to estimate terms on the right-hand side of \eqref{11.2-1} by using \eqref{11.1} and \eqref{11.2}. We deduce the following estimate
	\begin{align}
		\label{11.3}
		[u]^q_{W^{\sigma_i,q}(B_{\rho_i})}&\le C\frac{\left( C(N,p)2^{\frac{iN}{p-1}}\mathcal{T}+[u]_{W^{\sigma_{i-1},q}(B_{\rho_{i-1}})}+1\right)^q 2^{i(N+2q+2)} }{(R-r)^{N+2q+2}}\nonumber\\
		&\le C\frac{2^{iN+i(N+2q+2)}\left( \mathcal{T}+[u]_{W^{\sigma,q}(B_{\rho_{i-1}})}+1\right)^q }{(R-r)^{N+2q+2}}.
	\end{align}
	From Remark \ref{remC}, we know that the constant $C$ in \eqref{11.3} can be estimated from above, i.e., we have
	\begin{align*}
		C\le C_*:=\frac{\tilde{C}(N,p,q)}{s_1s_p(1-s_1)\left[ N-\frac{q}{q-p+2}\left( N+s_pp-2-\sigma p+2\sigma\right) \right]^\frac{q-p+2}{q}\sigma^2\left( \frac{s_pp}{p-1}-\gamma\right) }.
	\end{align*}
	Iterating the inequality \eqref{11.3}, we obtain
	\begin{align}
		\label{11.4}
		[u]^q_{W^{\sigma_i,q}(B_{\rho_i})}\le \frac{(2i+1)C^i_*2^{(Nq+N+2q+2)\left( \sum\limits_{j=1}\limits^{i}j\right) }\left( \mathcal{T}+[u]_{W^{\sigma,q}(B_R)}+1\right)^q }{(R-r)^{(N+2q+2)i}}
	\end{align}
	for any $i\in \mathbb{N}$. We denote by $i_0\in \mathbb{N}$ the smallest integer such that $\sigma_{i_0}\ge\gamma$. In other words, note that
	\begin{align*}
		\sigma_i=(\sigma-\tilde{\gamma})\left( 1-\frac{p-1}{q}\right)^i+\tilde{\gamma}, 
	\end{align*}
	we have
	\begin{align*}
		i_0:=\left[ \frac{\ln\frac{\tilde{\gamma}-\gamma}{\tilde{\gamma}-\sigma}}{\ln \frac{q-p+1}{q}}\right]+1\le \frac{\ln\frac{2}{\frac{s_pp}{p-2}-\gamma}}{\ln\frac{q-p+1}{q}}+1.
	\end{align*}
	For the norm $[u]_{W^{\gamma,q}(B_r)}$, \eqref{11.4} implies
	\begin{align*}
		[u]^q_{W^{\gamma,q}(B_r)}&\le 2^q[u]^q_{W^{\sigma_{i_0},q}(B_{\rho_{i_0}})}\\
		&\le\frac{C_*^{i_0}(2i_0+1)2^{(Nq+N+2q+2)\left( \sum\limits_{j=1}\limits^{i_0}j\right) }\left( \mathcal{T}+[u]_{W^{\sigma,q}(B_R)}+1\right)^q }{(R-r)^{(N+2q+2)i_0}}.
	\end{align*}
	By letting $C=C_*^{i_0}(2i_0+1)2^{(Nq+N+2q+2)\left( \sum\limits_{j=1}\limits^{i_0}j\right) }$ and $\kappa =(N+2q+2)i_0$, we conclude the claim.
\end{proof}

In the rest of this section, we shall prove the Sobolev regularity result as stated in Theorem \ref{th12}, i.e., the almost $W^{\frac{s_pp}{p-1},q}_\mathrm{loc}$-regularity of weak solutions for any $q\ge p$. Consequently, this result implies H\"{o}lder regularity of weak solutions by an application of Lemma \ref{lemMorrey1}. Furthermore, in the proof of Theorem \ref{th12}, we will utilize the fact that the weak solution $u$ is locally bounded in $\Omega$, this fact enables us to obtain 
\begin{align*}
	u\in W^{\gamma,m}_\mathrm{loc}(\Omega)\Rightarrow u\in W^{\frac{\gamma m}{q},q}_\mathrm{loc}(\Omega)
\end{align*}
for $q>m\ge p$.

\begin{proof}[Proof of Theorem \ref{th12}] We first consider the case $s_p\in \left( 0,\frac{2}{p}\right] $. By reason of $u\in L^\infty_\mathrm{loc}(\Omega)$, we have $u\in W^{\frac{s_pp}{q},q}_\mathrm{loc}(\Omega)$ and the following estimate
	\begin{align*}
		[u]^q_{W^{\frac{s_pp}{q},q}(B_R)}
		&= \int_{B_R}\int_{B_R}\frac{|u(x)-u(y)|^{q-p}|u(x)-u(y)|^p}{|x-y|^{N+s_pp}}\,dxdy\\
		&\le 2^{q-p}\|u\|^{q-p}_{L^\infty(B_R)}[u]^p_{W^{s_p,p}(B_R)}\\
		&\le C \left( [u]_{W^{s_p,p}(B_R)}+\|u\|_{L^\infty(B_R)}\right)^q. 
	\end{align*}
	From Lemma \ref{lem11}, we set $\sigma=\frac{s_pp}{q}>\mu^{s_p}_p$ to obtain
	\begin{align*}
		[u]^q_{W^{\gamma,q}(B_r)}&\le \frac{C\left( \mathcal{T}+[u]_{W^{\frac{s_pp}{q},q}(B_R)}+1\right)^q }{(R-r)^\kappa}\\
		&\le \frac{C\left( \mathcal{T}+[u]_{W^{s_p,p}(B_R)}+1\right)^q}{(R-r)^\kappa}.
	\end{align*}
	
	In the case $s_p\in \left( \frac{2}{p},\frac{p-1}{p}\right]$, this will happen when $p>3$. Note that $\max\left\lbrace \mu^{s_p}_p,0\right\rbrace=\mu^{s_p}_p$, for $i\in\mathbb{N}$, we define
	\begin{align*}
		\rho_i:=r+\frac{1}{2^i}(R-r),\quad\mathcal{T}_i:=\|u\|_{L^\infty(B_{\rho_i})}+\mathrm{Tail}(u;x_0,\rho_i)
	\end{align*}
	and the sequence of exponents
	\begin{align*}
		q_i:=\alpha^ip
	\end{align*}
	for some $\alpha\in \left( 1,\frac{s_p(p-2)}{s_pp-2}\right) $ that will be determined later. From now on, we are going to prove
	\begin{align}
		\label{12.1}
		[u]^{q_i}_{W^{\gamma,q_i}(B_{\rho_i})}\le \frac{\tilde{C}_i\left( \mathcal{T}+[u]_{W^{s_p,p}(B_R)}+1\right)^{q_i}}{(R-r)^{\tilde{\kappa}_i}}
	\end{align}
	for any $i\in\mathbb{N}$ by an induction argument.
	
	For $i=0$, since $s_p>\mu^{s_p}_p$, we apply Lemma \ref{lem11} with $(\sigma,q,r,R)$ replaced by $(s_p,p,\rho_0,R)$ to deduce
	\begin{align*}
		[u]^p_{W^{\gamma,p}(B_{\rho_0})}\le \frac{\tilde{C}_0\left( \mathcal{T}+[u]_{W^{s_p,p}(B_R)}+1\right)^{p}}{(R-r)^{\tilde{\kappa}_0}}.
	\end{align*}
	
	For the induction step we assume that $\text{\eqref{12.1}}_i$ holds true, which implies $[u]_{W^{\gamma,q_i}(B_{\rho_i})}<+\infty$. One can obtain $u\in W^{\gamma\frac{q_i}{q_{i+1}},q_{i+1}}(B_{\rho_i})$ since $u$ is locally bounded in $\Omega$ and
	\begin{align*}
		[u]^{q_{i+1}}_{W^{\gamma\frac{q_i}{q_{i+1}},q_{i+1}}(B_{\rho_i})}&=\int_{B_{\rho_i}}\int_{B_{\rho_i}}\frac{|u(x)-u(y)|^{q_{i+1}-q_i}|u(x)-u(y)|^{q_i}}{|x-y|^{\gamma q_i}}\,dxdy\\
		&\le [u]^{q_{i+1}}_{W^{\gamma,q_i}(B_{\rho_i})}+2^{q_{i+1}}\|u\|^{q_{i+1}}_{L^\infty(B_R)}.
	\end{align*}
	Observe that $\frac{\gamma}{\alpha}>\mu^{s_p}_p$, we can apply Lemma \ref{lem11} to have
	\begin{align}
		\label{12.2}
		[u]^{q_{i+1}}_{W^{\gamma,q_{i+1}}(B_{\rho_{i+1}})}&\le C_{i+1}\left( \mathcal{T}+[u]_{W^{\gamma\frac{q_i}{q_{i+1}},q_{i+1}}(B_{\rho_i})}+1\right)^{q_{i+1}}\frac{2^{i\kappa_{i+1}}}{(R-r)^{\kappa_{i+1}}}\nonumber\\
		&\le 2^{q_{i+1}}C_{i+1}\left( \mathcal{T}+[u]_{W^{\gamma,q_{i}}(B_{\rho_i})}+1\right)^{q_{i+1}}\frac{2^{i\kappa_{i+1}}}{(R-r)^{\kappa_{i+1}}}\nonumber\\
		&\le 2^{q_{i+1}}C_{i+1}\left( \mathcal{T}+\frac{\tilde{C}_i^\frac{1}{q_i}\left( \mathcal{T}+[u]_{W^{s_p,p}(B_R)}+1\right) }{(R-r)^{\frac{\tilde{\gamma}_i}{q_i}}}+1\right)^{q_{i+1}}\frac{2^{i\kappa_{i+1}}}{(R-r)^{\kappa_{i+1}}}\nonumber\\
		&\le 2^{2q_{i+1}}C_{i+1}\tilde{C}_i^\alpha\left( \mathcal{T}+[u]_{W^{s_p,p}(B_R)}+1\right) ^{q_{i+1}}\frac{2^{i\kappa_{i+1}}}{(R-r)^{\kappa_{i+1}+\tilde{\kappa}_{i}\alpha}}\nonumber\\
		&\le \frac{\tilde{C}_{i+1}\left( \mathcal{T}+[u]_{W^{s_p,p}(B_R)}+1\right) ^{q_{i+1}}}{(R-r)^{\tilde{\kappa}_{i+1}}},
	\end{align}
	where we set
	\begin{align*}
		\tilde{C}_{i+1}:=2^{2q_{i+1}}C_{i+1}\tilde{C}_i^\alpha
	\end{align*}
	and
	\begin{align*}
		\tilde{\kappa}_{i+1}:=\kappa_{i+1}+\tilde{\kappa}_{i}\alpha
	\end{align*}
	in the last line.
	Thus, we prove $\text{\eqref{12.1}}_{i+1}$ and finish the induction argument.
	
	When $q=p$, our claim is $\text{\eqref{12.1}}_{0}$. For the case $q>p$, we choose $i_0\in \mathbb{N}$ as the smallest integer such that
	\begin{align*}
		\left( \frac{s_p(p-2)}{s_pp-2}\right)^{i_0}>\frac{q}{p}. 
	\end{align*}
	That is,
	\begin{align*}
		i_0:=\left[ \frac{\ln\frac{q}{p}}{\ln\frac{s_p(p-2)}{s_pp-2}}\right]+1.
	\end{align*}
	We can define $\alpha:=\left( \frac{q}{p}\right)^\frac{1}{i_0}\in \left( 1,\frac{s_p(p-2)}{s_pp-2}\right)$ and deduce the result from $\text{\eqref{12.1}}_{i_0}$.
\end{proof}

Now we can obtain the H\"{o}lder regularity result.

\begin{proof}[Proof of Corollary \ref{cor13}]
	Let $\tilde{\gamma}=\frac{1}{2}\left( \gamma+\frac{s_pp}{p-1}\right) $ and $q=\frac{2N}{\frac{s_pp}{p-1}-\gamma}$. We apply the Morrey type embedding Lemma \ref{lemMorrey1} with $\gamma$ replaced by $\tilde{\gamma}$ and the claim of Theorem \ref{th12} to obtain
	\begin{align*}
		[u]_{C^{0,\gamma}(B_r)}=[u]_{C^{0,\tilde{\gamma}-\frac{N}{q}}(B_r)}&\le C[u]_{W^{\tilde{\gamma},q}(B_r)}\le \frac{C\left( \mathcal{T}+[u]_{W^{s_p,p}(B_R)}+1\right) }{(R-r)^\kappa}
	\end{align*}
	for some $C$ and $\kappa$ depending on $N, p, q, s_1, s_p, \gamma$.
\end{proof}

\section{The case $s_p\in \left( \frac{p-1}{p},1\right) $}
\label{sec5}

In this section, we deal with the case with the parameter $s_p$ big. Our proof is divided into two steps. In the first step, we will show that the gradient of weak solution $u$ exists, i.e.,
\begin{align*}
	u\in W^{s_p,p}_\mathrm{loc}(\Omega)\Rightarrow u\in W^{1,p}_\mathrm{loc}(\Omega).
\end{align*}
In the second step, we are going to demonstrate the almost Lipschitz regularity result, that is,
\begin{align*}
	u\in W^{1,p}_\mathrm{loc}(\Omega)\Rightarrow u\in W^{1,q}_\mathrm{loc}(\Omega)
\end{align*}
for any $q\ge p$. To show the weak solution $u$ admits a gradient $\nabla u$, we remark that  we will first prove that $u\in W^{\alpha,p}_\mathrm{loc}(\Omega)$ for some $\alpha>p-s_pp$ in Lemma \ref{lem16} and then we will show the existence of $\nabla u$ in $\left( L^p_\mathrm{loc}(\Omega)\right) ^N$.
\begin{lemma}
	\label{lem15}
	Let $p\ge2$, $s_p\in \left( \frac{p-1}{p}, 1\right) $ and  let $u\in W^{\sigma,p}_\mathrm{loc}(\Omega)$ be a locally bounded weak solution to problem \eqref{1.1} in the sense of Definition \ref{def1} with $\sigma\in \left( \max\left\lbrace \mu^{s_p}_p,0\right\rbrace ,1\right) $. Then, in the case $\sigma\le p-s_pp$, we have
	\begin{align*}
		u\in W^{\alpha,p}_\mathrm{loc}(\Omega)
	\end{align*}
	for any $\alpha\in (\sigma, \beta)$, where
	\begin{align*}
		\beta=\frac{\sigma}{p}+s_p.
	\end{align*}
	Moreover, there exists a constant $C=C(N,p,s_1,s_p,\sigma,\alpha)$ such that for any ball $B_R\equiv B_R(x_0)\subset\subset\Omega$ with $R\in (0,1)$ and any $r\in (0,R)$,
	\begin{align*}
		[u]^p_{W^{\alpha,p}(B_r)}\le \frac{C\left( \mathcal{T}+[u]_{W^{s_p,p}(B_R)}+1\right)^p }{(R-r)^{N+2p+2}}.
	\end{align*}
\end{lemma}
\begin{proof}
	We begin our proof with the following inequality that can be obtained in the proof of Lemma \ref{lem10},
	\begin{align}
		\label{15.1}
		\int_{B_r}|\tau_h(\tau_hu)|^p\,dx\le \frac{C\left( \mathcal{T}+[u]_{W^{\sigma,p}(B_R)}+1\right)^p }{(R-r)^{N+2p+2}}|h|^{s_pp+\sigma}
	\end{align}
	for any $h\in B_d\backslash\left\lbrace 0\right\rbrace $ with $d=\frac{1}{7}(R-r)$.
	
	Since $s_pp+\sigma\le p$, one can see that
	\begin{align*}
		\beta=\frac{\sigma}{p}+s_p\le1.
	\end{align*} 
	Thus, from \eqref{15.1} and Lemma \ref{lem6}, we have, in the case $\beta=1$,
	\begin{align}
		\label{15.2}
		\int_{B_r}|\tau_hu|^p\,dx&\le C|h|^{\tilde{\alpha} p}\left[ \frac{\left( \mathcal{T}+[u]_{W^{\sigma,p}(B_R)}+1\right)^p}{(R-r)^{N+2p+2}(1-\tilde{\alpha})}(R-r)^{(1-\alpha)p}+\frac{R^N\|u\|_{L^\infty(B_R)}}{(R-r)^{\tilde{\alpha}p}}\right] \nonumber\\
		&\le \frac{C\left( \mathcal{T}+[u]_{W^{\sigma,p}(B_R)}+1\right)^p}{(R-r)^{N+2p+2}}|h|^{\tilde{\alpha}p}
	\end{align}
	for any $h\in B_{\frac{1}{2}d}\backslash\left\lbrace 0\right\rbrace $, where $\tilde{\alpha}=\frac{1+\alpha}{2}$. In the case $\beta<1$, one can deduce the following inequality from \eqref{15.1} and Lemma \ref{lem6},
	\begin{align}
		\label{15.3}
		\int_{B_r}|\tau_hu|^p\,dx&\le C|h|^{\beta p}\left[ \frac{\left( \mathcal{T}+[u]_{W^{\sigma,p}(B_R)}+1\right)^p }{(1-\beta)^p(R-r)^{N+2p+2}}+\frac{R^N\|u\|^p_{L^\infty(B_R)}}{(R-r)^{\beta p}}\right] \nonumber\\
		&\le\frac{C\left( \mathcal{T}+[u]_{W^{\sigma,p}(B_R)}+1\right)^p}{(R-r)^{N+2p+2}}|h|^{\beta p}
	\end{align}
	for any $h\in B_{\frac{1}{2}d}\backslash\left\lbrace 0\right\rbrace $.
	
	Therefore, Lemma \ref{lem7} enables us to derive
	\begin{align*}
		\int_{B_r}\int_{B_r}\frac{|u(x)-u(y)|^p}{|x-y|^{N+\alpha p}}\,dxdy\le C\left[ \frac{\left( \mathcal{T}+[u]_{W^{\sigma,p}(B_R)}+1\right)^p}{(\beta-\alpha)(R-r)^{N+2p+2}}+\frac{R^N\|u\|^p_{L^\infty(B_R)}}{\alpha(R-r)^{\alpha p}}\right] 
	\end{align*}
	from \eqref{15.2} and \eqref{15.3}, where $\tilde{\alpha}-\alpha=\frac{1}{2}(\beta-\alpha)$ is taken into consideration.
	
	Hence, after some suitable calculation, we prove the claim
	\begin{align*}
		\int_{B_r}\int_{B_r}\frac{|u(x)-u(y)|^p}{|x-y|^{N+\alpha p}}\,dxdy\le \frac{C\left( \mathcal{T}+[u]_{W^{\sigma,p}(B_R)}+1\right)^p}{(R-r)^{N+2p+2}}.
	\end{align*}
\end{proof}

In the next lemma, we show that $\nabla u\in \left( L^p_\mathrm{loc}(\Omega)\right) ^N$ as long as $u\in W^{\alpha,p}_\mathrm{loc}(\Omega)$ for some $\alpha$ big.
\begin{lemma}
	\label{lem15.2}
	Let $p\ge2$, $s_p\in \left( \frac{p-1}{p},1\right) $ and let $u\in W^{\sigma,p}_\mathrm{loc}(\Omega)$ be a locally bounded weak solution to problem \eqref{1.1} in the sense of Definition \ref{def1} with $\sigma\in \left( \max\left\lbrace \mu^{s_p}_p,0\right\rbrace ,1\right) $. Then, in the case $\sigma> p-s_pp$, we have
	\begin{align*}
		u\in W^{1,p}_\mathrm{loc}(\Omega)
	\end{align*}
	and there exists a constant $C=C(N, p, s_1, s_p, \sigma)$ such that, for any ball $B_R\equiv B_R(x_0)\subset\subset\Omega$ with $R\in(0,1)$ and any $r\in (0,R)$,
	\begin{align*}
		\|\nabla u\|^p_{L^p(B_r)}\le \frac{C\left( \mathcal{T}+[u]_{W^{\sigma,p}(B_R)}+1\right)^p}{(R-r)^{N+2p+2}}.
	\end{align*}
\end{lemma}

\begin{proof}
	We can easily note that \eqref{15.1} also holds in the case $\sigma>p-s_pp$ for any $h\in B_d\backslash\left\lbrace 0\right\rbrace $ with $d=\frac{1}{7}(R-r)$. Since $s_p+\frac{\sigma}{p}>1$, applying Lemma \ref{lem6} we have
	\begin{align*}
		\int_{B_r}|\tau_hu|^p\,dx&\le C|h|^{p}\left[ \frac{\left( \mathcal{T}+[u]_{W^{\sigma,p}(B_R)}+1\right)^p(R-r)^{\left( s_p+\frac{\sigma}{p}-1\right) p}}{\left( s_p+\frac{\sigma}{p}-1\right) p}+\frac{\int_{B_R}|u|^p\,dx}{(R-r)^p}\right] \\
		&\le C|h|^{p}\left[ \frac{\left( \mathcal{T}+[u]_{W^{\sigma,p}(B_R)}+1\right)^p(R-r)^{\left( s_p+\frac{\sigma}{p}-1\right) p}}{\left( s_p+\frac{\sigma}{p}-1\right) p}+\frac{R^N\|u\|^p_{L^\infty(B_R)}}{(R-r)^p}\right] \\
		&\le \frac{C\left( \mathcal{T}+[u]_{W^{\sigma,p}(B_R)}+1\right)^p}{(R-r)^{N+2p+2}}|h|^p
	\end{align*} 
	for any $h\in B_{\frac{1}{2}d}\backslash\left\lbrace 0\right\rbrace $.
	
	Therefore, replacing $d$ by $R-r$ in Lemma \ref{lem14}, we conclude the claim
	\begin{align*}
		\int_{B_r}|\nabla u|^p\,dx\le \frac{C\left( \mathcal{T}+[u]_{W^{\sigma,p}(B_R)}+1\right)^p}{(R-r)^{N+2p+2}}.
	\end{align*}
\end{proof}

In the case $s_p\in \left( \frac{p-1}{p},\frac{p}{p+1}\right] $, we iterate the estimates obtained in Lemma \ref{lem15} to increase the power of the step size $|h|$.
\begin{lemma}
	\label{lem16}
	Let $p\ge2$, $s_p\in \left( \frac{p-1}{p},\frac{p}{p+1}\right] $ and let $u$ be a locally bounded weak solution to problem \eqref{1.1} in the sense of Definition \ref{def1}. Then, we have
	\begin{align*}
		u\in W^{\alpha,p}_\mathrm{loc}(\Omega)
	\end{align*}
	for some $\alpha\in (0,1)$ that satisfies $\alpha>p-s_pp$. Moreover, there exist two constants $C$ and $\kappa$ depending on $N, p, s_1, s_p, \alpha$, such that for any ball $B_R\equiv B_R(x_0)\subset\subset\Omega$ with $R\in(0,1)$ and any $r\in(0,R)$,
	\begin{align*}
		[u]^p_{W^{\alpha,p}(B_r)}\le \frac{C\left( \mathcal{T}+[u]_{W^{s_p,p}(B_R)}+1\right)^p}{(R-r)^{\kappa}}.
	\end{align*}
\end{lemma}

\begin{proof}
	Similar to the proof of Lemma \ref{lem11}, we set
	\begin{align*}
		\sigma_0:=s_p,\quad\sigma_{i+1}:=\frac{\sigma_i}{p}+\frac{p-1}{p}\tilde{\alpha}
	\end{align*}
	with $\tilde{\alpha}=\frac{1+\alpha}{2}$ for $i\in\mathbb{N}$ and
	\begin{align*}
		\mathcal{T}_i:=\|u\|_{L^\infty(B_{\rho_i})}+\mathrm{Tail}(u;x_0,\rho_i).
	\end{align*}
	Thus, we have $\sigma_i\rightarrow\tilde{\alpha}$ as $i\rightarrow+\infty$. Moreover, we define two sequences$\left\lbrace \rho_i\right\rbrace $, $\left\lbrace \beta_i\right\rbrace $ as
	\begin{align*}
		\rho_i:=r+\frac{1}{2^i}(R-r),\quad\beta_i:=\frac{\sigma_{i-1}}{p}+s_p.
	\end{align*}
	By applying Lemma \ref{lem15} with $(\alpha,\sigma,\beta,r,R)$ being replaced by $(\sigma_{i-1},\sigma_i,\beta_i,\rho_i,\rho_{i-1})$, we obtain
	\begin{align}
		\label{16.1}
		[u]^p_{W^{\sigma,p}(B_{\rho_i})}\le \frac{C_i\left( \mathcal{T}_i+[u]_{W^{s_p,p}(B_R)}+1\right)^p}{(\rho_{i-1}-\rho_i)^{N+2p+2}}.
	\end{align}
	This application is valid provided that $\sigma_i<1$, $\beta_i\le1$ and $[u]_{W^{\sigma_{i-1},p}(B_{\rho_{i-1}})}<+\infty$.
	
	For the terms in the right-hand side of \eqref{16.1}, we utilize estimate \eqref{11.1} and \eqref{11.2} to deduce
	\begin{align}
		\label{16.2}
		[u]^p_{W^{\sigma_i,p}(B_{\rho_i})}\le \frac{C_i2^{iNp+i(N+2p+2)}\left( \mathcal{T}+[u]_{W^{\sigma_{i_1},p}(B_{\rho_{i-1}})}+1\right)^p}{(R-r)^{N+2p+2}}.
	\end{align}
	Iterating the inequality \eqref{16.2}, we have
	\begin{align}
		\label{16.3}
		[u]^p_{W^{\sigma_i,p}(B_{\rho_i})}\le \frac{(2i+1)C^i_*2^{(Np+N+2p+2)\left( \sum\limits_{j=1}^{i}j\right) }\left( \mathcal{T}+[u]_{W^{s_p,p}(B_{R})}+1\right)^p}{(R-r)^{(N+2p+2)i}}
	\end{align}
	for any $i\in\mathbb{N}$ satisfying $\sigma_i<1$, where $C_i\le C_*$ with $C_*$ is independent of the parameter $i$.
	
	We now denote by $i_0\in\mathbb{N}$ a fixed integer such that $\sigma_{i_0}\ge\alpha$, i.e.,
	\begin{align*}
		i_0:=\left[ \frac{\ln\frac{\tilde{\alpha}-s_p}{\tilde{\alpha}-\alpha}}{\ln p}\right]+1 .
	\end{align*}
	For the norm $[u]_{W^{\alpha,p}(B_r)}$, \eqref{16.3} implies
	\begin{align*}
		[u]^p_{W^{\alpha,p}(B_r)}&\le 2^p[u]^{p}_{W^{\sigma_0,p}(B_{\rho_{i_0}})}\\
		&\le \frac{(2i_0+1)C^{i_0}_*2^{(Np+N+2p+2)\left( \sum\limits_{j=1}^{i_0}j\right) }\left( \mathcal{T}+[u]_{W^{s_p,p}(B_{R})}+1\right)^p}{(R-r)^{(N+2p+2)i_0}}.
	\end{align*}
	By letting $C=C_*^{i_0}(2i_0+1)2^{(Np+N+2p+2)\left( \sum\limits_{j=1}^{i_0}j\right) }$ and $\kappa=(N+2p+2)i_0$, we complete this proof.
\end{proof}

\begin{proposition}
	\label{pro17}
	Let $p\ge2$, $s_p\in\left( \frac{p-1}{p},1\right) $ and let $u$ be a locally bounded weak solution to problem \eqref{1.1} in the sense of Definition \ref{def1}. Then, we have
	\begin{align*}
		u\in W^{1,p}_\mathrm{loc}(\Omega).
	\end{align*}
	Moreover, there exist two constants $C$ and $\kappa$ depending on $N, p, s_1, s_p$, such that for any ball $B_R\equiv B_R(x_0)\subset\subset\Omega$ with $R\in(0,1)$ and any $r\in (0,R)$,
	\begin{align*}
		\|\nabla u\|^p_{L^p(B_r)}\le \frac{C\left( \mathcal{T}+[u]_{W^{s_p,p}(B_{R})}+1\right)^p}{(R-r)^\kappa}.
	\end{align*}
\end{proposition}

\begin{proof}
	We first consider the case $s_p\in \left( \frac{p}{p+1},1\right) $. In this case, note that $s_p>p-s_p$, we apply Lemma \ref{lem15.2} directly with $\sigma=s_p$ to obtain
	\begin{align*}
		\|\nabla u\|^p_{L^p(B_r)}\le \frac{C\left( \mathcal{T}+[u]_{W^{s_p,p}(B_{R})}+1\right)^p}{(R-r)^{N+2p+2}},
	\end{align*}
	where we set $\kappa=N+2p+2$.
	
	Next, we consider the case $s_p\in \left( \frac{p-1}{p},\frac{p}{p+1}\right] $. In this case, we apply Lemma \ref{lem16} with $r$ replaced by $\tilde{r}:=r+\frac{1}{2}(R-r)$ to have
	\begin{align}
		\label{17.1}
		[u]^p_{W^{\alpha,p}(B_{\tilde{r}})}\le \frac{C\left( \mathcal{T}+[u]_{W^{s_p,p}(B_{R})}+1\right)^p}{(R-r)^{\tilde{\kappa}}}
	\end{align}
	for some $\alpha>p-s_pp$ with $C,\kappa$ depending on $N,p,s_1,s_p,\alpha$. Then, we apply Lemma \ref{lem15.2} to conclude
	\begin{align*}
		\|\nabla u\|^p_{L^p(B_r)}&\le \frac{C\left( \mathcal{T}+[u]_{W^{s_p,p}(B_{\tilde{r}})}+1\right)^p}{(R-r)^{N+2p+2}}\\
		&\le \frac{C\left( \mathcal{T}+\frac{\mathcal{T}+[u]_{W^{s_p,p}(B_{R})}+1}{(R-r)^{\frac{\tilde{\kappa}}{p}}}+1\right)^p}{(R-r)^{N+2p+2}}\\
		&\le \frac{C\left( \mathcal{T}+[u]_{W^{s_p,p}(B_{R})}+1\right)^p}{(R-r)^{N+2p+2+\tilde{\kappa}}}.
	\end{align*}
We complete this proof by setting
	\begin{align*}
		\kappa=
		\begin{cases}
			N+2p+2+\tilde{\kappa}&\text{ if }s_p\in \left( \frac{p-1}{p},\frac{p}{p+1}\right],\\[2mm]
			 N+2p+2&\text{ if }s_p\in \left( \frac{p}{p+1},1\right).
		\end{cases}
	\end{align*}
\end{proof}

Given the established result  $\nabla u\in \left( L^p_\mathrm{loc}(\Omega)\right) ^N$, the Sobolev embedding  $W^{1, p}\hookrightarrow W^{s_p, p}$ directly yields the following corollary.

\begin{corollary}
	\label{pro19}
	Let $p\ge2$, $s_1,s_p\in(0,1)$ and $q\ge p$. There exists a constant $C=C(N, p, q, s_1, s_p)$ such that for any $u\in W^{1,q}_\mathrm{loc}(\Omega)$ being a locally bounded weak solution to problem \eqref{1.1} in the sense of Definition \ref{def1}, we have, for any ball $B_R\equiv B_R(x_0)\subset\subset\Omega$ with $R\in(0,1)$, $r\in(0,R)$, $d\in \left( 0,\frac{1}{4}(R-r)\right] $ and $h\in B_d\backslash\left\lbrace 0\right\rbrace $,
	\begin{align*}
		&\quad\int_{B_r}\int_{B_r}\frac{\left| J_{\frac{q}{p}+1}(\tau_hu)(x)-J_{\frac{q}{p}+1}(\tau_hu)(y) \right|^p }{|x-y|^{N+s_pp}}\,dxdy\\
		&\le C\frac{R^{1-s_1}}{R-r}\int_{B_R}|\tau_hu|^{q-p+1}\,dx+C\frac{1}{s_1(R-r)^{s_1}}\int_{B_R}|\tau_hu|^{q-p+1}\,dx\\
		&\quad+C\frac{R^{(N-\varepsilon)\frac{q-p+2}{q}}}{(R-r)^2}\|\nabla u\|^{p-2}_{L^q(B_{R+d})}\left( \int_{B_R}|\tau_hu|^{q}\,dx\right)^{\frac{q-p+2}{q}}\\
		&\quad+C\frac{\mathcal{T}^{p-2}}{(R-r)^{N+s_pp}}\int_{B_R}|\tau_hu|^{q-p+2}\,dx+C\frac{\mathcal{T}^{p-1}|h|}{(R-r)^{N+s_pp+1}}\int_{B_R}|\tau_hu|^{q-p+1}\,dx,
	\end{align*}
	where
	\begin{align*}
		\mathcal{T}:=\|u\|_{L^\infty(B_{R+d}(x_0))}+\mathrm{Tail}(u;x_0,R+d).
	\end{align*}
\end{corollary}

\begin{proof}
	Since $u\in W^{1,q}_\mathrm{loc}(\Omega)$, one can apply Lemma \ref{lem18}, the embedding $W^{1,q}\hookrightarrow W^{\sigma,q}$ to obtain
	\begin{align}
		\label{19.1}
		[u]^q_{W^{\sigma,q}(B_{R+d})}\le C(N)\frac{R^{(1-\sigma)q}}{(1-\sigma)q}\int_{B_R}|\nabla u|^q\,dx.
	\end{align}
	By bringing \eqref{19.1} into the claim in Proposition \ref{pro9}, we get the desired result.
\end{proof}

By Corollary \ref{pro19}, one can expect the fractional Sobolev regularity of $\nabla u$ under the assumption that $u\in W^{1,q}_\mathrm{loc}(\Omega)$. Then Lemma \ref{lem24} follows immediately from Lemma \ref{lem23}.
\begin{lemma}
	\label{lem22}
	Let $q\ge2$, $s_p\in\left( \frac{p-1}{p},1\right) $ and let $u\in W^{1,q}_\mathrm{loc}(\Omega)$ be a locally bounded weak solution to problem \eqref{1.1} in the sense of Definition \ref{def1}, we have
	\begin{align*}
		u\in W^{1+\alpha,q}_\mathrm{loc}(\Omega)
	\end{align*}
	for any $\alpha\in \left( 0,\frac{s_pp-(p-1)}{q}\right) $. Moreover, there exists a constant $C=C(N, p, q, s_1, s_p, \alpha)$ such that for any ball $B_R\equiv B_R(x_0)\subset\subset\Omega$ with $R\in(0,1)$ and any $r\in(0,R)$,
	\begin{align*}
		[\nabla u]^{q}_{W^{\alpha,q}(B_r)}\le \frac{C\left( \mathcal{T}+\|u\|_{L^q(B_R)}+1\right)^q}{(R-r)^{N+2q+2}}.
	\end{align*}
\end{lemma}

\begin{proof}
	We first apply Lemma \ref{lem20} with the parameter $R$ being replaced by $\tilde{R}:=\frac{9}{10}R+\frac{1}{10}r$ and $d:=\frac{1}{10}(R-r)$ to have
	\begin{align}
		\label{22.1}
		\|\tau_hu\|_{L^q(B_{\tilde{R}})}\le C(N)|h|\|\nabla u\|_{L^q(B_{R})}.
	\end{align}
	Moreover, the application of Corollary \ref{pro19} enables us to obtain
	\begin{align*}
		&\quad\int_{B_{r+5d}}\int_{B_{r+5d}}\frac{\left| J_{\frac{q}{p}+1}(\tau_hu)(x)-J_{\frac{q}{p}+1}(\tau_hu)(y)\right|^p  }{|x-y|^{N+s_pp}}\,dxdy\\
		&\le \frac{C}{R-r}\left( \int_{B_{\tilde{R}}}|\tau_hu|^{q}\,dx\right)^{\frac{q-p+1}{q}}+\frac{C}{s_1(R-r)^{s_1}}\left( \int_{B_{\tilde{R}}}|\tau_hu|^{q}\,dx\right)^{\frac{q-p+1}{q}}\\
		&\quad+\frac{C}{(R-r)^2}\|\nabla u\|^{p-2}_{L^q(B_{R})}\left( \int_{B_{\tilde{R}}}|\tau_hu|^{q}\,dx\right)^{\frac{q-p+2}{q}}\\
		&\quad+\frac{C\mathcal{T}^{p-2}}{(R-r)^{N+s_pp}}\left( \int_{B_{\tilde{R}}}|\tau_hu|^{q}\,dx\right)^{\frac{q-p+2}{q}}+\frac{C\mathcal{T}^{p-1}|h|}{(R-r)^{N+s_pp+1}}\left( \int_{B_{\tilde{R}}}|\tau_hu|^{q}\,dx\right)^{\frac{q-p+1}{q}},
	\end{align*}
	where H\"{o}lder's inequality is utilized to generate the term $\int_{B_{\tilde{R}}}|\tau_hu|^{q}\,dx$.
	
Substituting \eqref{22.1} into the above inequality, we deduce
	\begin{align}
		\label{22.2}
		&\quad\int_{B_{r+5d}}\int_{B_{r+5d}}\frac{\left| J_{\frac{q}{p}+1}(\tau_hu)(x)-J_{\frac{q}{p}+1}(\tau_hu)(y)\right|^p  }{|x-y|^{N+s_pp}}\,dxdy\nonumber\\
		&\le \frac{C}{R-r}\|\nabla u\|^{q-p+1}_{L^q(B_R)}|h|^{q-p+1}+\frac{C}{s_1(R-r)^{s_1}}\|\nabla u\|^{q-p+1}_{L^q(B_R)}|h|^{q-p+1}\nonumber\\
		&\quad+\frac{C}{(R-r)^2}\|\nabla u\|^{q}_{L^q(B_{R})}|h|^{q-p+2}+\frac{C\mathcal{T}^{p-2}}{(R-r)^{N+s_pp}}\|\nabla u\|^{q-p+2}_{L^q(B_R)}|h|^{q-p+2}\nonumber\\
		&\quad+\frac{C\mathcal{T}^{p-1}|h|}{(R-r)^{N+s_pp+1}}\|\nabla u\|^{q-p+1}_{L^q(B_R)}|h|^{q-p+2}\nonumber\\
		&\le \frac{C\left( \mathcal{T}+\|\nabla u\|_{L^q(B_R)}+1\right)^q }{(R-r)^{N+s_pp+1}}|h|^{q-p+1}.
	\end{align}
	For the left-hand side of \eqref{22.2}, Lemma \ref{lem5} implies
	\begin{align}
		\label{22.3}
		&\quad\int_{B_{r+4d}}\left| \tau_\lambda\left( J_{\frac{q}{p}+1}(\tau_hu)\right) \right| ^p\,dx\nonumber\\
		&\le C|\lambda|^{s_pp}\left[ \left[ J_{\frac{q}{p}+1}(\tau_hu)\right]^p_{W^{s_p,p}(B_{r+5d})}+\frac{\int_{B_{r+5d}}|\tau_hu|^q\,dx}{(R-r)^p} \right] \nonumber\\
		&\le C|\lambda|^{s_pp}\left[ \frac{\left( \mathcal{T}+\|\nabla u\|_{L^q(B_R)}+1\right)^q }{(R-r)^{N+s_pp+1}}|h|^{q-p+1}+\frac{\|\nabla u\|^q_{L^q(B_R)}}{(R-r)^{p}}|h|^q\right]\nonumber\\
		&\le \frac{C\left( \mathcal{T}+\|\nabla u\|_{L^q(B_R)}+1\right)^q|\lambda|^{s_pp}|h|^{q-p+1}}{(R-r)^{N+s_pp+1}}.
	\end{align}
	We choose $\lambda=h$ and note that
	\begin{align*}
		C(p,q)\left| \tau_\lambda\left( J_{\frac{q}{p}+1}(\tau_hu)\right) \right|\ge |\tau_h(\tau_hu)|^\frac{q}{p}.
	\end{align*}
	Therefore, for any $h\in B_d\backslash\left\lbrace 0\right\rbrace $, it holds that
	\begin{align*}
		\int_{B_{r+4d}}|\tau_h(\tau_hu)|^q\,dx\le \frac{C\left( \mathcal{T}+\|\nabla u\|_{L^q(B_R)}+1\right)^q|h|^{s_pp+q-p+1}}{(R-r)^{N+s_pp+1}}.
	\end{align*}
	
	As a result of application of Lemma \ref{lem21}, we conclude
	\begin{align*}
		\nabla u\in \left( W^{\alpha,q}(B_R)\right)^N 
	\end{align*}
	for any $\alpha\in \left( 0,\frac{s_pp-(p-1)}{q}\right)$ and there exists a constant $C=C(N, p, q, s_1, s_p, \alpha)$ such that
	\begin{align*}
		[\nabla u]^{q}_{W^{\alpha,q}(B_r)}&\le C\left[ \frac{\left( \mathcal{T}+\|\nabla u\|_{L^q(B_R)}+1\right)^q}{(R-r)^{N+s_pp+1}}+\frac{\|\nabla u\|^q_{L^q(B_R)}}{(R-r)^{s_pp+q-p+1}}\right] \\
		&\le \frac{C\left( \mathcal{T}+\|u\|_{L^q(B_R)}+1\right)^q}{(R-r)^{N+2q+2}}.
	\end{align*}
\end{proof}

\begin{lemma}
	\label{lem24}
	Let $q\ge2$, $s_p\in \left( \frac{p-1}{p},1\right) $, $q\ge p$ and let $u\in W^{1,q}_\mathrm{loc}(\Omega)$ be a locally bounded weak solution to problem \eqref{1.1} in the sense of Definition \ref{def1}. Then, we have
	\begin{align*}
		\nabla u\in \left( L^\frac{Nq}{N-\alpha q}_{\mathrm{loc}}(\Omega)\right)^N 
	\end{align*}
	for any $\alpha\in \left( 0,\frac{s_pp-(p-1)}{q}\right) $. Moreover, there exists a constant $C=C(N, p, q, s_1, s_p, \alpha)$ such that for any ball $B_R\equiv B_R(x_0)\subset\subset\Omega$ with $R\in(0,1)$ and any $r\in (0,R)$,
	\begin{align*}
		\left(\int_{B_r}|\nabla u|^\frac{Nq}{N-\alpha q}\,dx\right) ^{\frac{N-\alpha q}{N}}\le \frac{C\left( \mathcal{T}+\|\nabla u\|_{L^q(B_R)}+1\right)^q}{(R-r)^{N+2q+2}}.
	\end{align*}
\end{lemma}

\begin{proof}
	For any $\alpha\in \left( 0,\frac{s_pp-(p-1)}{q}\right)$, the direct application of Lemma \ref{lem22} and Lemma \ref{lem23} enables us to obtain
	\begin{align*}
		\frac{1}{R^{N-\alpha q}}\left(\int_{B_r}|\nabla u|^\frac{Nq}{N-\alpha q}\,dx\right)^\frac{N-\alpha q}{N}&\le CR^{\alpha q-N}\int_{B_R}\int_{B_R}\frac{|\nabla u(x)-\nabla u(y)|^q}{|x-y|^{N+\alpha q}}\,dxdy\\ 
		&\le \frac{CR^{\alpha q-N}\left( \mathcal{T}+\|\nabla u\|_{L^q(B_R)}+1\right)^q}{(R-r)^{N+2q+2}}.
	\end{align*}
	That is,
	\begin{align*}
		\left(\int_{B_r}|\nabla u|^\frac{Nq}{N-\alpha q}\,dx\right) ^{\frac{N-\alpha q}{N}}\le \frac{C\left( \mathcal{T}+\|\nabla u\|_{L^q(B_R)}+1\right)^q}{(R-r)^{N+2q+2}}.
	\end{align*}
\end{proof}

\begin{lemma}
	\label{lem25}
	Let $p\ge2$, $s_p\in\left( \frac{p-1}{p}, 1\right) $ and let $u\in W^{1,p}_\mathrm{loc}(\Omega)$ be a locally bounded weak solution of problem \eqref{1.1} in the sense of Definition \ref{def1}, we have
	\begin{align*}
		u\in W^{1,q}_\mathrm{loc}(\Omega)
	\end{align*}
	for any $q\ge p$. Moreover, there exist $C$ and $\kappa$ depending on $N, p, q, s_1, s_p$ such that for any ball $B_R\equiv B_R(x_0)\subset\subset\Omega$ with $R\in(0, 1)$ and any $r\in(0,R)$,
	\begin{align*}
		\|\nabla u\|^q_{L^q(B_r)}\le \frac{C\left( \mathcal{T}+\|\nabla u\|_{L^p(B_R)}+1\right)^q}{(R-r)^\kappa}.
	\end{align*}
\end{lemma}

\begin{proof}
	We prove this claim by an iteration argument. First, we define
	\begin{align*}
		q_i:=\left( \frac{N}{N-\frac{1}{2}\left( s_pp-p+1\right) }\right) ^ip
	\end{align*}
	and
	\begin{align*}
		\rho_i:=r+\frac{1}{2^i}(R-r)
	\end{align*}
	for $i\in\mathbb{N}$.
	
	For a fixed $i\in\mathbb{N}$, suppose that $\|\nabla u\|_{L^{q_{i-1}}(B_{\rho_{i-1}})}<+\infty$, we apply Lemma \ref{lem24} with parameter $(r,R,q,\alpha)$ being replaced by $(\rho_i,\rho_{i-1},q_{i-1},\frac{s_pp-p+1}{2q_{i-1}})$ to have
	\begin{align*}
		\left( \int_{B_{\rho_i}}|\nabla u|^{q_i}\,dx\right)^\frac{q_{i-1}}{q_i}&\le\frac{C_i\left( \mathcal{T}_{i-1}+\|\nabla u\|_{L^{q_{i-1}}(B_{\rho_{i-1}})}+1\right)^{q_{i-1}} }{(\rho_{i-1}-\rho_{i})^{N+2q_{i-1}+2}} \\
		&\le \frac{C_i\left( \mathcal{T}_{i-1}+\|\nabla u\|_{L^{q_{i-1}}(B_{\rho_{i-1}})}+1\right)^{q_{i-1}} }{(R-r)^{N+2q_{i-1}+2}}2^{i(N+2q+2)}\\
		&\le \frac{C_i2^{i(Nq_i+N+2q+2)}\left( \mathcal{T}+\|\nabla u\|_{L^{q_{i-1}}(B_{\rho_{i-1}})}+1\right)^{q_{i-1}} }{(R-r)^{N+2q_{i-1}+2}}\\
		&=\frac{\tilde{C}_i\left( \mathcal{T}+\|\nabla u\|_{L^{q_{i-1}}(B_{\rho_{i-1}})}+1\right)^{q_{i-1}} }{(R-r)^{N+2q_{i-1}+2}},
	\end{align*}
	where $\tilde{C}_i=C_i2^{i(Nq_i+N+2q+2)}$. Iterating the above inequality, one has
	\begin{align*}
		\left( \int_{B_{r}}|\nabla u|^{q_i}\,dx\right)^\frac{1}{q_i}\le  \frac{C\left( \mathcal{T}+\|\nabla u\|_{L^{q_{i-1}}(B_{\rho_{i-1}})}+1\right)}{(R-r)^\kappa},
	\end{align*}
	where we set $C=\prod\limits_{j=1}^i\tilde{C}_j$ and $\kappa=(N+2)i+2\sum\limits_{j=0}^{i-1}q_j$.
	
	Since $q_i\rightarrow +\infty$ as $i\rightarrow+\infty$, for any $q\ge p$, there exists $i_0\in\mathbb{N}$ such that $q_{i_0}\ge q$. More precisely, by choosing
	\begin{align*}
		i_0:=\left[ \frac{\ln\frac{q}{p}}{\ln\frac{N}{N-\frac{1}{2}(s_pp-p+1)}}\right]+1, 
	\end{align*}
	and using  H\"{o}lder's inequality to obtain
	\begin{align*}
		\left( \int_{B_r}|\nabla u|^q\,dx\right)^\frac{1}{q}\le \left( \int_{B_r}|\nabla u|^{q_{i_0}}\,dx\right)^\frac{1}{q_{i_0}}\le\frac{C\left( \mathcal{T}+\|\nabla u\|_{L^p(B_R)}+1\right)}{(R-r)^{\kappa_{i_0}}}
	\end{align*}
	with $C=\prod\limits_{j=1}^{i_0}\tilde{C}_j$.
\end{proof}

Having the previous lemmas in hand, it suffices to prove the almost Lipschitz regularity of weak solutions.

\begin{proof}[Proof of Theorem \ref{th26}]
	We set
	\begin{align*}
			\mathcal{T}_{\frac{R+r}{2}}:=\|u\|_{L^\infty(B_{\frac{R+r}{2}}(x_0))}+\mathrm{Tail}\left(u;x_0,\frac{R+r}{2}\right).
	\end{align*}
	By combining the following inequalities that come from Lemma \ref{lem25} and Proposition \ref{pro17} respectively,
	\begin{align*}
		\|\nabla u\|^q_{L^q(B_r)}\le \frac{C_1\left( \mathcal{T}_{\frac{R+r}{2}}+\|\nabla u\|_{L^p(B_{\frac{R+r}{2}})}+1\right)^q }{(R-r)^{\kappa_1}}
	\end{align*}
	and
	\begin{align*}
		\|\nabla u\|^p_{L^p(B_{\frac{R+r}{2}})}\le \frac{C_2\left( \mathcal{T}+[u]_{W^{s_p,p}(B_R)}+1\right)^p }{(R-r)^{\kappa_2}},
	\end{align*}
	one can conclude the claim of Theorem \ref{th26}, i.e.,
	\begin{align*}
		\|\nabla u\|^q_{L^q(B_r)}&\le \frac{C_1\left( \mathcal{T}_{\frac{R+r}{2}}+\frac{C_2^\frac{1}{p}\left( \mathcal{T}+[u]_{W^{s_p,p}(B_R)}+1\right) }{(R-r)^{\frac{\kappa_2}{p}}}+1\right)^q }{(R-r)^{\kappa_1}}\\
		&\le \frac{(2^Nq+3^q)C_1C_2^\frac{q}{p}\left( \mathcal{T}+[u]_{W^{s_p,p}(B_R)}+1\right)^q}{(R-r)^{\kappa_1+\frac{\kappa_2 q}{p}}}.
	\end{align*}
This completes the proof.
\end{proof}

A direct application of Lemma \ref{Morrey2} enables us to prove Corollary \ref{cor27} from Theorem \ref{th26}.

\begin{proof}[Proof of Corollary \ref{cor27}]
	From Theorem \ref{th26}, we know that the locally bounded weak solution $u\in W^{1,q}_\mathrm{loc}(\Omega)$ for any $q\ge p$ provide that $p\ge2$ and $s_p\in\left( \frac{p-1}{p},1\right) $. Thus, by Morrey's embedding Lemma \ref{Morrey2}, we have $u\in C^{0,\gamma}_\mathrm{loc}(\Omega)$ for any $\gamma\in (0,1)$.
	
	To be specific, for any fixed $\gamma\in(0,1)$, let us apply Lemma \ref{Morrey2} with parameter $q:=\frac{N}{1-\gamma}$ and Theorem \ref{th26} to conclude the following estimate
	\begin{align*}
		[u]_{C^{0,\gamma}(B_r)}= [u]_{C^{0,1-\frac{N}{q}}(B_r)}\le C\|\nabla u\|_{L^q(B_r)}\le \frac{C\left( \mathcal{T}+[u]_{W^{s_p,p}(B_R)}+1\right) }{(R-r)^{\kappa}}.
	\end{align*}
\end{proof}

\subsection*{Acknowledgments}
This work was supported by the National Natural Science Foundation of China (No. 12471128).

\subsection*{Conflict of interest} The authors declare that there is no conflict of interest. We also declare that this manuscript has no associated data.
	
\subsection*{Data availability}
No data was used for the research described in the article.


\begin{thebibliography}{[a]}
	
		
		
		\bibitem{AMRT08} F. Andreu, J.~M. Maz\'on, J.~D. Rossi and J.~J. Toledo, A nonlocal $p$-Laplacian evolution equation with Neumann boundary conditions, J. Math. Pures Appl. (9) 90 (2) (2008) 201--227.
		
		\bibitem{AMRT09} F. Andreu, J.~M. Maz\'on, J.~D. Rossi and J.~J. Toledo, A nonlocal $p$-Laplacian evolution equation with nonhomogeneous Dirichlet boundary conditions, SIAM J. Math. Anal. 40 (5) (2008) 1815--1851.
		
		\bibitem{A83} G. Anzellotti, Pairings between measures and bounded functions and compensated compactness, Ann. Mat. Pura Appl. 135 (4) (1983) 293--318.
		
		\bibitem{BM20} L. Beck and G. Mingione, Lipschitz bounds and nonuniform ellipticity, Comm. Pure Appl. Math. 73 (5) (2020) 944--1034.

		\bibitem{BDL2401} V. B\"{o}gelein, F. Duzaar, N. Liao, G. M. Bisci and R. Serbadei, Gradient regularity for $(s,p)$-harmonic functions, preprint, 2024, arXiv:2409.02012.
		
		\bibitem{BDL2402} V. B\"{o}gelein, F. Duzaar, N. Liao, G. M. Bisci and R. Serbadei, Regularity for the fractional $p$-Laplace equation, preprint, 2024, arXiv:2406.01568.
		
		\bibitem{BL17} L. Brasco and E. Lindgren, Higher Sobolev regularity for the fractional $p$-Laplace equation in the superquadratic case, Adv. Math. 304 (2017) 300--354.
		
		\bibitem{BLS18} L. Brasco, E. Lindgren and A. Schikorra, Higher H\"older regularity for the fractional $p$-Laplacian in the superquadratic case, Adv. Math. 338 (2018) 782--846.
		
		\bibitem{B23} C. Bucur, Solutions of the fractional 1-Laplacian: existence, asymptotics and flatness results, NoDEA Nonlinear Differential Equations Appl. 32 (3) (2025) Paper No. 52.
		
		\bibitem{BDLM23} C. Bucur, S. Dipierro, L. Lombardini and J.~M. Maz\'on, $(s, p)$-harmonic approximation of functions of least $W^{s,1}$-seminorm, Int. Math. Res. Not. IMRN (2) (2023) 1173--1235.
		
		\bibitem{BDLV25} C. Bucur, S. Dipierro, L. Lombardini and E. Valdinoci, Continuity of $s$-minimal functions, Calc. Var. Partial Differential Equations 64 (2) (2025) Paper No. 66, 17 pp.
		
		\bibitem{C18} I. Chlebicka, A pocket guide to nonlinear differential equations in Musielak–Orlicz spaces, Nonlinear Anal. 175 (2018) 1--27.
		
		\bibitem{CF99} G.-Q. Chen and H. Frid, Divergence-measure fields and hyperbolic conservation laws, Arch. Ration. Mech. Anal. 147 (2) (1999) 89--118.
		
		\bibitem{CF01} G.-Q. Chen and H. Frid, On the theory of divergence-measure fields and its applications, Bol. Soc. Brasil. Mat. 32 (3) (2001) 401--433.
		
		\bibitem{CF03} G.-Q. Chen and H. Frid, Extended divergence-measure fields and the Euler equations for gas dynamics, Comm. Math. Phys. 236 (2) (2003) 251--280.
		
		\bibitem{CT03} M. Cicalese and C. Trombetti, Asymptotic behaviour of solutions to $p$-Laplacian equation, Asymptot. Anal. 35 (1) (2003) 27--40.
		
		\bibitem{C17} M. Cozzi, Regularity results and Harnack inequalities for minimizers and solutions of nonlocal problems: a unified approach via fractional De Giorgi classes, J. Funct. Anal. 272 (11) (2017) 4762--4837.
		
		\bibitem{CMM24} G. Cupini, P. Marcellini and E. Mascolo, Regularity for nonuniformly elliptic equations with $p,q$-growth and explicit $x,u$-dependence, Arch. Ration. Mech. Anal. 248 (4) (2024) Paper No. 60, 45 pp.
		
		\bibitem{DDP24} C. De~Filippis, F. De~Filippis and M. Piccinini, Bounded minimizers of double phase problems at nearly linear growth, preprint, 2024, arXiv:2411.14325.
		
		\bibitem{DM2301} C. De~Filippis and G. Mingione, Nonuniformly elliptic Schauder theory, Invent. Math. 234 (3) (2023) 1109--1196.
		
		\bibitem{DM23} C. De~Filippi and G. Mingione, Regularity for double phase problems at nearly linear growth, Arch. Ration. Mech. Anal. 247 (5) (2023) Paper No. 85, 50 pp.
		
		\bibitem{DM24} C. De~Filippis and G. Mingione, Gradient regularity in mixed local and nonlocal problems, Math. Ann. 388 (1) (2024) 261--328.
		
		\bibitem{DP24} F. De~Filippis and M. Piccinini, Regularity for multi-phase problems at nearly linear growth, J. Differential Equations 410 (2024) 832--868.
		
		\bibitem{DFZ24} M. Ding, Y. Fang and C. Zhang, Local behavior of the mixed local and nonlocal problems with nonstandard growth, J. Lond. Math. Soc. (2) 109 (6) (2024) Paper No. e12947, 34 pp.
		
		\bibitem{DKP16} A. Di~Castro, T. Kuusi and G. Palatucci, Local behavior of fractional {$p$}-minimizers, Ann. Inst. H. Poincar\'e{} C Anal. Non Lin\'eaire 33 (5) (2016) 1279--1299.
		
		\bibitem{DKLN23} L. Diening, K. Kim, H. Lee and S. Nowak, Higher differentiability for the fractional $p$-Laplacian, Math. Ann. 391 (4) (2024) 5631--5693.
		
		\bibitem{DPV12} E. Di~Nezza, G. Palatucci and E. Valdinoci, Hitchhiker's guide to the fractional Sobolev spaces, Bull. Sci. Math. 136 (5) (2012) 521--573.
		
		\bibitem{D04} A. Domokos, Differentiability of solutions for the non-degenerate $p$-Laplacian in the Heisenberg group, J. Differential Equations  204 (2) (2004) 439--470.
		
		\bibitem{DL76} G. Duvaut and J.-L. Lions, Inequalities in Mechanics and Physics, Grundlehren der Mathematischen Wissenschaften, vol. 219, Springer-Verlag, Berlin-New York, 1976.
		
		\bibitem{EM00} L. Esposito and G. Mingione, Partial regularity for minimizers of convex integrals with $L\log L$-growth, NoDEA Nonlinear Differential Equations Appl. 7 (1) (2000) 107--125.
		
		\bibitem{FLZ25} Y. Fang, D. Li and C. Zhang, Higher Sobolev regularity on the mixed local and nonlocal $p$-Laplace equations, preprint, 2025, arXiv:2501.09487.
		
		\bibitem{FS98} M. Fuchs and G.~A. Seregin, A regularity theory for variational integrals with $L\ln L$-growth, Calc. Var. Partial Differential Equations 6 (2) (1998) 171--187.
		
		\bibitem{FS99} M. Fuchs and G.~A. Seregin, Variational methods for fluids of Prandtl-Eyring type and plastic materials with logarithmic hardening, Math. Methods Appl. Sci. 22 (4) (1999) 317--351.
		
		\bibitem{GL24} P. Garain and E. Lindgren, Higher H\"older regularity for the fractional $p$-Laplace equation in the subquadratic case, Math. Ann. 390 (4) (2024) 5753--5792.
		
		\bibitem{GT22} Y. Giga and S. Tsubouchi, Continuity of derivatives of a convex solution to a perturbed one-Laplace equation by $p$-Laplacian, Arch. Ration. Mech. Anal. 244 (2) (2022) 253--292.
		
		\bibitem{GT83} D. Gilbarg and N.~S. Trudinger, Elliptic partial differential equations of second order, second edition, Grundlehren der mathematischen Wissenschaften  224, Springer, Berlin, 1983.
		
		\bibitem{G23} W. G\'{o}rny, Strongly anisotropic Anzellotti pairings and their applications to the anisotropic $p$-Laplacian, J. Math. Anal. Appl. (2025), 129734, https://doi.org/10.1016/j.jmaa.2025.129734.
		
		\bibitem{GMT24} W. G\'{o}rny, J.~M. Maz\'on and J.~J. Toledo, Evolution problems with perturbed 1-Laplacian type operators on random walk spaces, Math. Ann. (2025), https://doi.org/10.1007/s00208-025-03180-z.
		
		\bibitem{G15} G. Grubb, Fractional Laplacians on domains, a development of H\"ormander's theory of $\mu$-transmission pseudodifferential operators, Adv. Math. 268 (2015) 478--528.
		
		\bibitem{HM22} D. Hauer and J.~M. Maz\'on, The Dirichlet-to-Neumann operator associated with the 1-Laplacian and evolution problems, Calc. Var. Partial Differential Equations 61 (1) (2022) Paper No. 37, 50 pp.
		
		\bibitem{HS14} K. Ho and I. Sim, Existence and some properties of solutions for degenerate elliptic equations with exponent variable, Nonlinear Anal. 98 (2014) 146--164.
		
		\bibitem{K90} B. Kawohl, On a family of torsional creep problems, J. Reine Angew. Math. 410 (1990) 1--22.
		
		\bibitem{L19} P. Lindqvist, Notes on the stationary {$p$}-Laplace equation, SpringerBriefs in Mathematics, Springer, Cham, 2019. xi+104 pp.
		
		\bibitem{M89} P. Marcellini, Regularity of minimizers of integrals of the calculus of variations with nonstandard growth conditions, Arch. Ration. Mech. Anal. 105 (3) (1989) 267--284.
		
		\bibitem{M91} P. Marcellini, Regularity and existence of solutions of elliptic equations with $p, q$-growth conditions, J. Differential Equations 90 (1) (1991) 1--30.
		
		\bibitem{M21} P. Marcellini, Growth conditions and regularity for weak solutions to nonlinear elliptic pdes, J. Math. Anal. Appl. 501 (1) (2021) Paper No. 124408, 32 pp.
		
		\bibitem{MMT23} J.~M. Maz\'on, A. Molino and J.~J. Toledo, Doubly nonlinear equations for the 1-Laplacian, J. Evol. Equ. 23 (4) (2023) Paper No. 67, 26 pp.
		
		\bibitem{MPRT16} J.~M. Maz\'on, M. P\'erez-Llanos, J.~D. Rossi and J.~J. Toledo, A nonlocal 1-Laplacian problem and median values, Publ. Mat. 60 (1) (2016) 27--53.
		
		\bibitem{MRS14} J.~M. Maz\'on, J.~D. Rossi and S. Segura~de~Le\'on, Functions of least gradient and 1-harmonic functions, Indiana Univ. Math. J. 63 (4) (2014) 1067--1084.
		
		\bibitem{MRS15} J.~M. Maz\'on, J.~D. Rossi and S. Segura~de~Le\'on, The 1-Laplacian elliptic equation with inhomogeneous Robin boundary conditions, Differential Integral Equations 28 (5-6) (2015) 409--430.
		
		\bibitem{MRT16} J.~M. Maz\'on, J.~D. Rossi and J.~J. Toledo, Fractional $p$-Laplacian evolution equations, J. Math. Pures Appl. (9) 105 (6) (2016) 810--844.
		
		\bibitem{MRT19} J.~M. Maz\'on, J.~D. Rossi and J.~J. Toledo, Nonlocal perimeter, curvature and minimal surfaces for measurable sets, J. Anal. Math. 138 (1) (2019) 235--279.
		
		\bibitem{MST08} A. Mercaldo, S. Segura~de~Le\'on and C. Trombetti, On the behaviour of the solutions to $p$-Laplacian equations as $p$ goes to 1, Publ. Mat. 52 (2) (2008) 377--411.
		
		\bibitem{MST09} A. Mercaldo, S. Segura~de~Le\'on and C. Trombetti, On the solutions to 1-Laplacian equation with $L^1$ data, J. Funct. Anal. 256 (8) (2009) 2387--2416.
		
		\bibitem{MRST10} A. Mercaldo, J.~D. Rossi, S. Segura~de~Le\'on and C. Trombetti, Anisotropic $p, q$-Laplacian equations when $p$ goes to 1, Nonlinear Anal. 73 (11) (2010) 3546--3560.
		
		\bibitem{MR21} G. Mingione and V. Radulescu, Recent developments in problems with nonstandard growth and nonuniform ellipticity, J. Math. Anal. Appl. 501 (1) (2021) Paper No. 125197, 41 pp.
		
		\bibitem{MS99} G. Mingione and F. Siepe, Full $C^{1, \alpha}$-regularity for minimizers of integral functionals with $L\log L$-growth, Z. Anal. Anwendungen 18 (4) (1999) 1083--1100.
		
		\bibitem{NO23} M. Novaga and F. Onoue, Local H\"older regularity of minimizers for nonlocal variational problems, Commun. Contemp. Math. 25 (10) (2023) Paper No. 2250058, 29 pp.
		
		\bibitem{S93} H. Spohn, Surface dynamics below the roughening transition, Journal de Physique I 3 (1993) 69--81.
		
		\bibitem{T21} S. Tsubouchi, Local Lipschitz bounds for solutions to certain singular elliptic equations involving the one-Laplacian, Calc. Var. Partial Differential Equations 60 (1) (2021) Paper No. 33, 35 pp.
		
		\bibitem{T24} S. Tsubouchi, A weak solution to a perturbed one-Laplace system by $p$-Laplacian is continuously differentiable, Math. Ann. 388 (2) (2024) 1261--1322.
		
		\bibitem{T25} S. Tsubouchi, Gradient continuity for the parabolic $(1, p)$-Laplace equation under the subcritical case, Ann. Mat. Pura Appl. (4) 204 (1) (2025) 261--287.
		
		\bibitem{T24A} S. Tsubouchi, Gradient continuity for the parabolic $(1, p)$-Laplace system, preprint, 2025, arXiv:2503.16808.
		
		\bibitem{T25A} S. Tsubouchi, Continuity of the spatial gradient of weak solutions to very singular parabolic equations involving the one-Laplacian, preprint, 2025, arXiv:2306.06868.
		
		\bibitem{VZ10} V. Vergara and R. Zacher, A priori bounds for degenerate and singular evolutionary partial integro-differential equations, Nonlinear Anal. 73 (11) (2010) 3572--3585.
	\end{thebibliography}
\end{document}